\documentclass[11pt]{amsart}
\usepackage{latexsym, amsmath, amssymb, amsthm, mathrsfs}
\usepackage[alwaysadjust]{paralist}
\usepackage{caption}
\usepackage{subcaption}
\usepackage[T1]{fontenc}
\usepackage{lmodern}
\usepackage[backref,colorlinks=true,linkcolor=blue,urlcolor=blue, citecolor=blue]%
  {hyperref}
\usepackage[shortalphabetic]{amsrefs}
\usepackage[all]{xy}
\usepackage[T1]{fontenc}
\usepackage{mathptmx}
\usepackage{microtype}
\usepackage{framed}
\usepackage{tikz}
\usepackage{graphicx}

\usepackage[centering, includeheadfoot, hmargin=1in, vmargin=1in,
  headheight=30.4pt]{geometry}
\newtheorem{lemma}{Lemma}[section]
\newtheorem{theorem}[lemma]{Theorem}
\newtheorem{corollary}[lemma]{Corollary}
\newtheorem{proposition}[lemma]{Proposition}

\theoremstyle{plain}
\newtheorem*{theoremA}{Theorem A}

\newtheorem*{theoremC}{Theorem C}
\newtheorem*{theoremD}{Theorem D}
\newtheorem*{corollaryB}{Corollary B}
\theoremstyle{definition}
\newtheorem{definition}[lemma]{Definition}
\newtheorem{remark}[lemma]{Remark}
\newtheorem{example}[lemma]{Example}

\renewcommand{\theequation}%
{\arabic{section}.\arabic{lemma}.\arabic{equation}}

\newcommand{\CC}{\ensuremath{\mathbb{C}}} 
\newcommand{\NN}{\ensuremath{\mathbb{N}}} 
\newcommand{\PP}{\ensuremath{\mathbb{P}}} 
\newcommand{\QQ}{\ensuremath{\mathbb{Q}}} 
\newcommand{\RR}{\ensuremath{\mathbb{R}}} 
\newcommand{\ZZ}{\ensuremath{\mathbb{Z}}} 
 
\newcommand{\sI}{\ensuremath{\kern -1pt \mathscr{I}\kern -2pt}} 
\newcommand{\sJ}{\ensuremath{\kern -2pt \mathscr{J}\kern -2pt}} 
\newcommand{\sO}{\ensuremath{\mathscr{O}}}

\newcommand{\shk}{\ensuremath{\mathcal{K}}}
\newcommand{\shc}{\ensuremath{\mathcal{C}}}
\newcommand{\shu}{\ensuremath{\mathcal{U}}}
\newcommand{\shs}{\ensuremath{\mathcal{S}}}
\newcommand{\shd}{\ensuremath{\mathcal{D}}}

\renewcommand{\geq}{\geqslant}
\renewcommand{\leq}{\leqslant}

\DeclareMathOperator{\mult}{mult}
\DeclareMathOperator{\Nef}{Nef}

\DeclareMathOperator{\Supp}{Supp}
\DeclareMathOperator{\ord}{ord}
\DeclareMathOperator{\Null}{Null}
\DeclareMathOperator{\Neg}{Neg}
\DeclareMathOperator{\vol}{vol}
\DeclareMathOperator{\length}{length}
\DeclareMathOperator{\Sing}{Sing}
\DeclareMathOperator{\Defect}{Defect}
\DeclareMathOperator{\Bbig}{Big}
\DeclareMathOperator{\intr}{int}

\newcommand{\equ}{\ensuremath{\,=\,}}
\newcommand{\dgeq}{\ensuremath{\,\geq\,}}
\newcommand{\deq}{\ensuremath{\stackrel{\textrm{def}}{=}}}

\newcommand{\dsubseteq}{\ensuremath{\,\subseteq\,}}
\newcommand{\st}[1]{\ensuremath{\left\{ #1 \right\}   }}
            
\begin{document}

\title[Local positivity]{Local positivity of linear series on surfaces}

\author[A.~K\" uronya]{Alex K\" uronya}
\author[V.~Lozovanu]{Victor Lozovanu}

\address{Alex K\"uronya, Budapest University of Technology and Economics, Department of Algebra, Egry J\'ozsef u. 1., H-111 Budapest, Hungary}
\email{{\tt alex.kuronya@math.bme.hu}}

\address{Victor Lozovanu, Universit\'a degli Studi di Milano--Bicocca, Dipartimento di Matematica e Applicazioni, Milano, Italy}
\email{{\tt victor.lozovanu@unimib.it}}

%


\maketitle


\section*{Introduction}

The main purpose of this paper is to understand local positivity of line bundles on surfaces, by making use of the theory of Newton--Okounkov bodies. More precisely, we find ampleness and nefness criteria in terms of convex geometry, and relate the information obtained this way to Seshadri-type 
invariants. 

For the past thirty-odd years there has been an increasing interest in describing positivity of line bundles around single points of varieties. Starting with the work of Demailly on Fujita's conjecture, where he introduced Seshadri constants, the topic developed quickly due to the work of 
Demailly,  Ein--Lazarsfeld, Nakamaye, and others \cites{D,EKL,ELMNP2,Nak1,Nak2}. In spite of all the effort, our understanding is still very limited, 
simple questions that remain unanswered to this day abound. 

We aim at translating existing invariants of local positivity to the language of Newton--Okounkov bodies, thus enriching them with extra structure. 
Our model is the relationship between Newton--Okounkov bodies and the volume of a divisor: we intend to replace numbers functioning as invariants 
by collections of convex bodies. We would like to emphasize the special nature of linear series on surfaces, most of the time we will employ 
elementary surface-specific tools, Zariski decomposition will play a crucial role for instance.

Originating in the work of Khovanskii from the late 70's, and Okounkov's construction \cite{Ok} in representation theory, Newton--Okounkov bodies are a not-quite-straightforward generalization of Newton polytopes to the setting of arbitrary projective varieties. For an $n$-dimensional 
projective variety $X$, a full flag of subvarieties $Y_\bullet$, and a big divisor $D$ on $X$, the Newton--Okounkov body $\Delta_{Y_\bullet}(D)\subseteq\RR^n$ is a convex set, encoding the set of all normalized valuation vectors coming from sections of 
multiples of $D$, where  the rank $n$ valuation of the function field of $X$ is determined by $Y_\bullet$. 

These convex bodies display  surprisingly good properties, and gave rise to a flurry of activities in projective geometry, combinatorics, and representation theory. 
For detailed descriptions and proofs the reader is kindly referred to the original sources \cites{KK,LM} and the recent review \cite{Bou}. 
The main idea is that  Newton--Okounkov bodies capture the vanishing behaviour of all sections of all multiples of $D$ at the same time. 

Our philosophical starting point is the work of Jow \cite{Jow} observing that the collection of all Newton--Okounkov bodies attached to a given line bundle serves as a universal numerical invariant: whenever all the Newton--Okounkov bodies agree for two divisors $D$ and $D'$, then in fact they must be numerically equivalent. Jow's result  leads to the expectation that one  should be able 
to read off numerical properties  of line bundles from the collection of attached Newton--Okounkov bodies. 

Once we focus on local positivity, the above principle modifies in the following manner: we expect that the local positivity of a Cartier divisor $D$ at a point $x\in X$ to be governed by the collection
\[
\shk(D;x) \deq  \st{ \Delta_{Y_\bullet}(D)\subseteq \RR^n_+\mid Y_\bullet \text{ is an admissible flag centered at $x$}} \ .
\]
It turns out that, whenever $X$ is a smooth projective surface, the above collection depends only on the choice of $x$ to a limited extent. In fact, it can be shown that $\shk(D;x)$ is independent of $x$ away from a countable union of proper subvarieties. 

We always work on smooth projective surfaces over  the complex number field, unless otherwise mentioned. Our first main result is a version  of the combinatorial characterization of torus-invariant ample/nef divisors on toric varieties valid on  all surfaces. We go one step further  and offer an analogous description in terms of infinitesimal Newton--Okounkov bodies. 

Let us  introduce some notation. For  positive real number $\lambda,\lambda', \xi >0$ we set 
\[
\Delta_{\lambda,\lambda'} \deq  \{(t,y)\in\RR^2_+ \ | \  \lambda't+\lambda y\leq \lambda\lambda'\}\ \ \ \text{and}\ \ \ 
\Delta_{\xi}^{-1} \deq \{(t,y)\in \RR^2_+ \ | \ 0\leq t\leq \xi, 0\leq y\leq t\}\ .
\] 
If $\lambda=\lambda'$, then denote by $\Delta_{\lambda}\deq \Delta_{\lambda,\lambda}$, which is the standard simplex of length $\lambda$.

We offer the following convex geometric characterization of  ampleness and nefness.

\begin{theoremA}\label{thm:A}(Nefness and ampleness criteria)
Let $X$ be a  smooth projective surface, $D$ a big $\RR$-divisor on $X$, and $\pi\colon X'\to X$ the blow-up of $x\in X$
with exceptional divisor $E$. Then 
\begin{enumerate}[i]
\item[(i)] $D$ is nef  $\Leftrightarrow$  for all $x\in X$ there exists a flag $(C,x)$ such that $(0,0)\in\Delta_{(C,x)}(D)$ \\
\hspace*{0.466in} $\Leftrightarrow$ for all $x\in X$ there exists $y\in E$ such that  $(0,0)\in\Delta_{(E,y)}(\pi^*(D))$.

\item[(ii)] $D$ is ample $\Leftrightarrow$ for all $x\in X$  there exists a flag $(C,x)$ and $\lambda >0$ such that $\Delta_{\lambda}\subseteq \Delta_{(C,x)}(D)$ \\
 \hspace*{0.679in} $\Leftrightarrow$ for all $x\in X$ there exists $y\in E$ and $\xi>0$ such that $\Delta_{\xi}^{-1}\subseteq\Delta_{(E,y)}(\pi^*(D))$.
\end{enumerate} 
\end{theoremA}

Theorem~A is a particular case of more general results. In Theorem~\ref{thm:main1} we prove the according criteria for a point $x\in X$ 
not to be contained either in the negative or null locus of $D$ in terms of Newton--Okounkov bodies defined on $X$. 
Furthermore, in Theorem~\ref{thm:triangle} we connect these loci to the shape of infinitesimal Newton--Okounkov bodies, 
defined on the blow-up $X'$.

As mentioned above, the Newton--Okounkov body of a big $\QQ$-divisor $D$ encodes how all the sections of all powers of $D$ vanish along a fixed flag.
Conversely, it is a very exciting problem to find out exactly which points in the plane are given by valuations of sections, whether these points 
lie in the interior or the boundary of the Newton--Okounkov body.
This is expressed by saying that a rational point of $\Delta_{Y_\bullet}(D)$ is 'valuative'. Finding valuative points in Newton--Okounkov bodies 
is a recurring theme of this article, some partial answers are summarized in the following corollary.

\begin{corollaryB}\label{cor:B}(Valuative points)
Let $X$ be a smooth projective surface, $D$ a big $\QQ$-divisor, $(C,x)$ a flag on $X$, and $\pi\colon X'\to X$ the blow-up of $X$ at $x$ 
with exceptional divisor $E$. Then
\begin{enumerate}
\item[(i)]  Any rational point in the interior of $\Delta_{(C,x)}(D)$ is valuative;
\item[(ii)] Suppose $\Delta_{\lambda,\lambda'}\subseteq \Delta_{(C,x)}(D)$ for some $\lambda,\lambda'>0$. Then any rational point on the horizontal segment $[0,\lambda)\times\{0\}$ and the vertical one $\{0\}\times [0,\lambda')$ is valuative;
\item[(iii)] Suppose that $\Delta^{-1}_{\xi} \dsubseteq \Delta_{(E,y_0)}(\pi^*(D))$
for some $\xi>0$ and $y_0\in E$. Then any rational point on the diagonal segment $\{(t,t)|0\leq t<\xi\}$ and on the horizontal segment  $[0, \xi)\times\{0\}$ is  valuative;
\end{enumerate} 
\end{corollaryB} 

It is interesting to note that statement $(ii)$ can be obtained via  restricted volumes (see \cite{ELMNP2} for the basic theory). However, we present here a different proof for the surface case, that relies only on ideas of convex geometric nature and Theorem~A. As for statement $(iii)$, the rational points on the diagonal are valuative due to the fact that the exceptional divisor $E$ is a rational curve.

As a consequence of Theorem~A, all Newton--Okounkov bodies of an ample divisor $A$ are bound to contain a standard simplex $\Delta_\lambda$ of some size. By  choosing the curve in the flag to be very positive, one can make the size of the simplex as small as we wish. Therefore the exciting question to
ask is how large it can become. This leads to the definition of the largest simplex constant:
\[
 \lambda(A;x) \deq \sup_{(C,x)}\sup\st{\lambda>0\mid \Delta_{\lambda}\subseteq \Delta_{(C,x)}(A)}\ , 
\]
where the first supremum runs through all admissible flags centered at the point $x\in X$. 

As expected of asymptotic invariants, the largest simplex constant is homogeneous  in $A$, invariant with respect to numerical equivalence of 
divisors; moreover it  is a super-additive function of $A$. In Proposition~\ref{prop:simplexseshadri} we observe that 
$\epsilon(A;x) \dgeq \lambda(A;x)$, where the left-hand side denotes the classical Seshadri constant of $A$ at $x$. 
We illustrate in Remark~\ref{rmk:fake} that  $\lambda(A;x)\neq \epsilon(A;x)$ in general. 

Using Diophantine approximation we establish the uniform positivity of largest simplex constants assuming the rational polyhedrality of the nef cone. 

\begin{theoremC}(Positivity of the largest simplex constant)
Let  $X$ be  a smooth projective surface  with a rational polyhedral nef  cone. Then there exists a strictly positive constant $\lambda(X)>0$
such that  
\[
\epsilon(A;x)\dgeq \lambda(A;x) \dgeq \lambda(X), 
\]
for any $x\in X$ and any ample Cartier divisor $A$ on $X$.
\end{theoremC}
Even further, we show the existence of strictly positive lower bound on Seshadri constants for any smooth projective variety $X$ of any dimension, whenever the nef cone of $X$ is rational polyhedral.

By turning our attention to the collection of infinitesimal Newton--Okounkov bodies we can take a closer look at Seshadri constants. As it happens, infinitesimal Newton--Okounkov bodies seem to capture  the local positivity of divisors more precisely. Hence one can expect them to determine (moving) Seshadri constants as well. This is indeed the case, as we will immediately explain.  

Here again, the key geometric invariant is the largest 'inverted' standard simplex that fits inside an infinitesimal Newton--Okounkov body. With notation as above, if $D$ is a big $\RR$-divisor on $X$, $\pi\colon\tilde{X}\to X$ the blow-up of $X$ at $x$ with exceptional
divisor $E$, then we set 
\[
\xi(\pi^*(D);y) \ \deq \  \sup\{\xi \geq 0 \mid \Delta_{\xi}^{-1}\subseteq \Delta_{(E,y)}(\pi^*(D))\} \ .
\]
It is not hard to see that this invariant does not depend on the choice of the point $y\in E$, as seen in Lemma~\ref{lemma:function constant}. So, if we denote by $\xi(D;x)\deq \xi(\pi^*(D);y)$ for some $y\in E$, then the following theorem says that this invariant is actually the moving Seshadri constant of $D$ at point $x$. 
\begin{theoremD}(Characterization of moving Seshadri constants) 
Let $D$ be a big $\RR$-divisor on $X$. If $x\notin\Neg(D)$, then 
\[
\epsilon(||D||;x)  \equ \xi(D;x) \ .
\]
\end{theoremD}

As a further application of infinitesimal Newton--Okounkov bodies, we translate ideas of Nakamaye and Cascini into the language of convex geometry 
to provide a new proof of the Ein--Lazarsfeld lower bound on Seshadri constants at very general points (see Corollary~\ref{cor:lowerbound}). 

A few words about the organization of this paper. Section 1 hosts   a quick recap of Newton--Okounkov bodies and Zariski decomposition, 
here several small new observations have been added that we will use repeatedly later on. Section 2 is devoted to our main results on 
Newton--Okounkov bodies on surfaces, while Section 3 is given over to the treatment of  infinitesimal Newton--Okounkov bodies and 
their relation to moving Seshadri  constants. In Section 4 we present various applications of the material developed so far.

\noindent
\textbf{Acknowledgments.} We are grateful for helpful discussions to S\'ebastian Boucksom, Lawrence Ein, John Christian Ottem, Mihnea Popa, 
and Stefano Urbinati. Parts of this work were done while the authors attended the MFO workshop on Newton--Okounkov bodies, the Summer School in Geometry at University of Milano--Bicocca, and the RTG Workshop on Newton--Okounkov bodies at the University of Illinois at Chicago. We would like to thank the organizers of the events for these opportunities
(Megumi Harada, Kiumars Kaveh, Askold Khovanskii; Francesco Bastianelli, Roberto Paoletti; Izzet Coskun and Kevin Tucker). 
 
 Alex K\"uronya was partially supported by the DFG-Forschergruppe 790 ``Classification of Algebraic Surfaces and Compact Complex Manifolds'', 
by the DFG-Graduier\-ten\-kol\-leg 1821 ``Cohomological Methods in Geometry'', and by the OTKA grants 77476 and 81203 of the Hungarian Academy of 
Sciences.

\section{Notation and introductory remarks}
We introduce the notation that will be used throughout this paper, and  give a brief introduction to Zariski decomposition and to the construction 
of Newton--Okounkov polygons for the sake of the reader. In addition we include some remarks for the lack of a suitable reference  
that had  probably been known to  experts.

\subsection{Zariski decomposition}
Let $X$ be a smooth complex projective surface.  As proven in \cite{F} (see also \cite{B01}*{Theorem~14.14} or \cite{Z} for the case of $\QQ$-divisors),
every  pseudo-effective $\RR$-divisor $D$ on $X$ has a \textit{Zariski decomposition}, i.e. $D$ can be written  uniquely  as a sum 
\[
D \ = \ P_D \ + \ N_D.
\]
of $\RR$-divisors, such that $P_D$ is nef, $N_D$ is either zero or an effective divisor with negative definite intersection matrix, and $(P_D\cdot N_D)=0$. The divisor $P_D$ is called the \textit{positive part},  $N_D$ the \textit{negative part} of $D$. Note that $P_D$ and $N_D$ will be $\QQ$-divisors whenever $D$ is such. Furthermore, when $D$ is a $\QQ$-divisor, Zariski decomposition is an equality of divisors and not merely of numerical equivalence classes. 

A crucial property of Zariski decomposition is that the positive part carries all the sections, more precisely 
(see \cite{PAG}*{Proposition 2.3.21} for instance), assuming that $mD$, $mP_D$ and $mN_D$ are all integral, the natural inclusion map $H^0(X,\sO_X(mP_D))\rightarrow H^0(X,\sO_X(mD))$, 
 defined by the multiplication with the divisor $mN_D$, is an isomorphism.

Following  \cite{BKS}, one can associate to $D$ the loci
\[
\Null (D) \deq \bigcup_{P_D.E=0} E \ \textup{ and } \Neg(D) \deq \bigcup_{E\subseteq \textup{Supp}(N_D)} E\ ,
\]
where the unions are taken  over irreducible curves on $X$. The orthogonality property of Zariski decomposition yields $\Neg(D)\subseteq \Null(D)$. 

In higher dimensions, these correspond to the  augmented and restricted base loci of $D$ introduced in \cite{ELMNP1}. 
We do not rely on the higher dimensional definition of these loci, but it is nevertheless important to keep in mind that 
$\textup{\textbf{B}}_+(D) = \Null(D)$ and $\textup{\textbf{B}}_-(D) = \Neg(D)$ as observed in \cite{ELMNP1}. 
Vaguely speaking  $\Null(D)$ consists of those points where $D$ is locally not ample while  $\Neg(D)$ is the locus of points where $D$ is locally not nef.

In \cite{BKS} the main goal is to prove the variation of Zariski decomposition inside the big cone. 
Based on the description of  Zariski chambers given there, we prove a statement  that can be seen as a more precise version of  Kodaira's lemma.

\begin{lemma}\label{lem:kodaira}
Let $P$ be a big and nef $\RR$-divisor on $X$, and let  $\Null(P)=E_1\cup\ldots \cup E_r$. 
Then there exists a maximal-dimensional rational polyhedral cone $\shc\subseteq\RR^r_{\geq 0}$ such that 
$P + \alpha\cdot (E_1,\dots,E_r)^T$ is ample for all $\alpha\in -\intr  \shc$ of sufficiently small norm. 
\end{lemma} 

\begin{remark}\label{rem:negativedefinite}
Under the assumptions of  Lemma~\ref{lem:kodaira}, it is an easy consequence of the Hodge index theorem that the intersection matrix of the curve 
$E_1,\ldots ,E_k$ is negative definite: take a non-trivial divisor $M=\sum_ia_iE_i$; notice that $(P\cdot M)=0$. Thus, by Hodge, we see that $M^2<0$,
since $M$ is not numerically trivial.  This last fact holds as the classes of the curves $E_1,\ldots ,E_k$ 
are linearly independent in $\textup{N}^1(X)_{\RR}$.
\end{remark}

\begin{proof}
The claim is immediate if $P$ is ample, hence we can assume  that $P\in\partial\Nef(X)\cap\Bbig(X)$. By \cite{BKS}*{Corollary 1.3}, 
there exists an open neighborhood $\shu$ of $P$ and the curves $C_1,\dots,C_s$ such that the boundary of
the Zariski chamber decomposition inside $\shu$ is given by the hyperplanes $C_i^\perp\subseteq \textup{N}^1(X)_\RR$. 

If $(P\cdot C_i)\neq 0$ for a curve $C_i$, then by shrinking $\shu$ if need be, we can arrange that $\shu$ lies on one side of $C_i^\perp$. Therefore, we can assume without loss of generality that 
\[
 \shu \ \cap  \ \textup{Amp}(X) \equ E_1^{>0} \ \cap \ \ldots \ \cap \  E_r^{>0} \ \cap \ \shu\ .
\]
By possibly shrinking $\shu$, we assume also that the self-intersection form is strictly positive on it. 

The Nakai--Moishezon criterion then says that a divisor $P-\sum_{i=1}^{r}a_iE_i\in \shu$ is ample if and only if 
\[
((P-\sum_{i=1}^{r}a_iE_i)\cdot E_j)\ >\ 0 \ \ \text{for every  $1\leq j\leq r$.}
\]
As $(P\cdot E_j)=0$, by assumption, for all $1\leq j\leq r$, then this is equivalent to 
\begin{equation}\label{equ:ampleness}
 \sum_{i=1}^{r}a_i(E_i\cdot E_j) > 0 \ \ \text{for all $1\leq j\leq r$.}
\end{equation}
We  are left with looking for a sufficient condition on $\alpha=(a_1,\dots,a_r)\in \RR_{\geq 0}^r$
such that (\ref{equ:ampleness}) holds; not surprisingly, this ends up being  a linear algebra question. 

Let $A$ be the intersection matrix of the curves $E_1,\dots,E_r$. This matrix is negative definite by Remark~\ref{rem:negativedefinite}. Then \cite{BKS}*{Lemma 4.1} shows that $A^{-1}$ is a non-singular matrix with non-positive entries. In this notation, the statement (\ref{equ:ampleness}) is then equivalent to ask that $A\cdot (a_1,\dots,a_r)^T$ has only strictly positive coefficients. 

Let $e_1,\dots,e_r$ be the standard basis of $\RR^r$ and $v_1\deq A^{-1}e_1,\dots,v_r\deq A^{-1}e_r$. We claim that the cone $\shc$ spanned by $v_1,\dots,v_r$ has the required property. Indeed, since the elements of $A^{-1}$ are all non-positive, then $-\shc\subseteq \RR^r_{\geq 0}$. Thus, every $v\in -\intr\shc$ has only positive coefficients. Furthermore, if $v=\sum_{i=1}^{r}a_iv_i\in \intr\shc$, then $Av=\sum_{i=1}^{r}a_iA(A^{-1}e_i)=\sum_{i=1}^{r}a_ie_i>0$, hence 
(\ref{equ:ampleness}) is satisfied. 
\end{proof}

\begin{remark}\label{rem:ZD of pullback}
Let $x\in X$ be a point and $\pi:X'\rightarrow X$ be the blow-up of $X$ at $x$. Suppose $D$ is a pseudo-effective $\RR$-divisor on $X$ and $D=P_D+N_D$ its Zariski decomposition. If $x\notin \Neg(D)$, then $\pi^*D = \pi^*P_D + \pi^*N_D$ is the Zariski decomposition of $\pi^*D$. To see this, it suffices, by uniqueness of Zariski decomposition, to check that the right-hand side has 
the right properties. So, $\pi^*P_D$ remains nef, and $(\pi^*P_D\cdot\pi^*N_D)=(P_D\cdot N_D)=0$. Since $x\notin \Neg(D)$, then $\pi^*N_D$ equals the strict transform of $N_D$, and the respective intersection
matrices agree. 
\end{remark}

\subsection{Newton--Okounkov polygons} 
For the general theory of Newton--Okounkov bodies the reader is kindly referred to \cite{KK} and \cite{LM} or the excellent expository work \cite{Bou}. Here we only summarize some surface-specific facts. 

As before, let $X$ be a smooth projective surface, and $D$ be a big $\QQ$-divisor on $X$. We say that the pair $(C,x)$ is an \textit{admissible flag} on $X$ if $C\subseteq X$ is an irreducible curve and $x\in C$ is a smooth point.  Then the \textit{Newton--Okounkov body} associated to this date is defined via  
\[
\Delta_{(C,x)}(D) \deq  \overline{\nu_{(C,x)}\big(\{D' \ | \ D'\sim_{\QQ}D \textup{ effective }\QQ-\textup{divisor}\}\big)} \ ,
\]
where the rank-two valuation $\nu_{(C,x)}(D')=(\nu_1(D'),\nu_2(D'))$ is given by 
\[
\nu_1(D') \equ \ord_C(D')\ \  \textup{ and }\ \ \nu_2(D') \equ \ord_x((D'-\nu_1(D')C)|_C)\ .
\]
Making use of  \cite{LM}*{Proposition~4.1},  one can replace $\QQ$-linear equivalence by  numerical equivalence (denoted by $\equiv$). If $D$ is a big $\RR$-divisor, one can also associate the Newton--Okounkov body $\Delta_{(C,x)}(D)$, where $\QQ$-linear equivalence is replaced by numerical equivalence in the definition.

Drawing on Zariski decomposition on surfaces, \cite{LM}*{Theorem 6.4} gives a practical  description of Newton--Okounkov bodies in the surface case. Following \cite{LM}*{Section 6}, let $\nu$ be the coefficient of $C$ in the negative part $N(D)$ and set
\[ 
\mu \ = \ \mu (D;C) \deq \sup\{ \ t>0 \ | \ D-tC \textup{ is big }\}\  . 
\]
Whenever  there is no risk of confusion we will write $\mu(D)$. 

For any $t\in [\nu ,\mu ]$ we set $D_t \deq D-tC$. Let  $D_t=P_t+N_t$ be the  Zariski decomposition of $D_t$. 
Consider the  functions $\alpha ,\beta \colon [\nu ,\mu ]\rightarrow \RR_+$ defined as follows
\[ 
\alpha(t) \deq \ord_x(N_t|_C)\ , \mbox{ }\mbox{ }\beta(t) \deq \ord_x(N_t|_C)+
P_t\cdot C\ .
\] 
Then Lazarsfeld and Musta\c t\u a show that the Newton--Okounkov body is described as follows.
\begin{theorem}[Lazarsfeld--Musta\c t\u a, \cite{LM}*{Theorem 6.4}] Let $D$ be a big $\RR$-divisor, and $(C,x)$ an admissible flag on a smooth projective surface. Then
\[ 
\Delta_{(C,x)} (D) \equ \{ (t,y)\in \RR^2_+ \ | \ \nu\leq t\leq \mu, \alpha(t)\leq y\leq \beta(t)\}\ . 
\]
\end{theorem}
\begin{remark}
The Newton--Okounkov body  $\Delta_{(C,x)}(D)\subseteq \RR^2_+$ has been shown to be a polygon in \cite{KLM}*{Section 2}. The results of  \cite{KLM} reveal further properties of  $\Delta_{(C,x)}(D)$. The function $t\longmapsto N_t$ is increasing on the interval $[\nu,\mu]\subseteq \RR$, i.e. for any $\nu\leq t_1\leq t_2\leq \mu$ the difference $N_{t_2}-N_{t_1}$ is an effective divisor. This implies that a vertex of  $\Delta_{(C,x)}(D)$ may only occur for those  $t\in [\nu,\mu]$, where a new curve appears in $\Neg(D_t)$. 
\end{remark}
\begin{remark}
It was observed in \cite{BKS}*{Proposition 1.14} that Zariski decomposition is continuous inside the big cone but not in general when the 
limiting divisor class is only pseudo-effective. It turns out that there exists another important situation where continuity holds. 

Let $C$ be an irreducible curve on $X$ and as before let $\mu(D)=\mu(D;C)$. From  \cite{KLM}*{Proposition 2.1} 
and the ideas of the proof of Proposition~1.14 from \cite{BKS}, it follows that 
\[
N_{D-tC}\to N_{D-\mu(D)C}\ \ \text{ and }\ \ P_{D-tC}\to P_{D-\mu(D)C}
\]
whenever $t\to \mu(D)$. If $x\in C$ is a smooth point, then this says in particular that the segment 
$\Delta_{(C,x)}(D)\cap \big(\mu(D)\times \RR\big)$ is computed from actual divisors on $X$, i.e. $P_{\mu(D)}$ and $N_{\mu(D)}$. 
\end{remark}
\begin{remark}\label{rem:shift}
The above description of Newton--Okounkov polygons gives rise to the equality 
\[
\Delta_{(C,x)}(D)_{t\geq t_0} \equ \Delta_{(C,x)}(D-t_0C) \ + \ (t_0,0) \ .
\]
for any $t_0\in\RR$. In  \cite{LM}*{Theorem 4.24}, this is proved  under the additional assumption that $C\nsubseteq \Supp(\Null(D))$.
As it turns out, this condition is not  necessary.
\end{remark}

\begin{remark}\label{rem:definition}
Going back to the initial definition of Newton--Okounkov polygons, one remarks that the valuation map can be seen through local intersection numbers. Let $D$ be a big $\QQ$-divisor, $(C,x)$ an admissible flag, and $D'\sim_{\QQ} D$ an effective divisor. Note that $C\nsubseteq \textup{Supp}(D'-\nu_1(D')C)$, thus 
\[
\nu_2(D') \equ \big((D'-\nu_1(D)C). C\big)_x\ ,
\]
where the right-hand side  is  the local intersection number of the effective $\QQ$-divisor $D-\nu_1(D)C$ and the curve $C$ at the point $x$.  
As it is well known, one has 
\[
(C'.C'') \equ \sum_{P\in X}(C'.C'')_P\ ,
\]
for distinct irreducible curves  $C',C''\subseteq X$. In particular, one has the inequality 
\[ 
\big((D-\nu_1(D')C).C\big) \ \geq \ \nu_2(D')
\]
as  $D\equiv D'$. These ideas give rise to a somewhat different construction of  Newton--Okounkov polygons in the surface case that is based on  local intersection numbers.
\end{remark}

\begin{remark}\label{rem:vertical}
For a Newton--Okounkov polygon $\Delta_{(C,x)}(D)$, the lengths of the vertical slices are independent of the point $x$. 
To see this, denote
\[
\Delta_{(C,x)}(D)_{t=\xi} \deq \Delta_{(C,x)}(D) \ \bigcap \ \{t=\xi\}\times\RR 
\]
for any $\xi\in [\nu,\mu]$. Then 
\[
\length(\Delta_{(C,x)}(D)_{t=\xi})  \equ \beta(\xi)-\alpha(\xi) \equ (P_{D-\xi C}.\cdot C)\ ,
\]
hence the observation. 
\end{remark}
The following lemma helps reduce the problem of computing the Newton--Okounkov polygon of a divisor to the computation of the polygon 
of its positive part. This is implicitly contained in \cite{LSS} for $\QQ$-divisors; here  we give  a proof of the $\RR$-divisor case.

\begin{lemma}\label{lem:positive}
Let $D$ be a big $\RR$-divisor on a smooth projective surface $X$ and  $(C,x)$ an admissible flag on $X$. If $x\notin \Neg(D)$, 
then 
\[
\Delta_{(C,x)}(D)=\Delta_{(C,x)}(P_D) \ .
\]
\end{lemma}
\begin{proof}
We first prove the statement for $\QQ$-divisors and then use  a continuity argument for the general case. If $D$ is a big $\QQ$-divisor, then by the homogeneity of  Newton--Okounkov polygons and  Zariski decomposition we can assume that $D$,$P_D$ and $N_D$ are all integral Cartier divisors. The multiplications maps  
\[
H^0(X,\sO_X(mP_D))\rightarrow H^0(X,\sO_X(mD)) 
\]
by $mN_D$ are isomorphisms (see \cite{PAG}*{Proposition~2.3.21}). Hence the  definition of  Newton--Okounkov polygons and the condition  $x\notin \Supp(N_D)$ imply the statement.

For the general case, fix a norm $\|\cdot \|$ on the finite-dimensional vector space $\textup{N}^1(X)_{\RR}$. Let $D$ be a big $\RR$-divisor, and $(A_n)_{n\in\NN}$ a sequence of ample $\RR$-divisors such that $\lim_{n\rightarrow \infty}\big(\|A_n\|\big)= 0$, 
$A_{n+1}-A_{n}$ is an ample $\QQ$-divisor, and $D+A_n$ is a $\QQ$-divisor for any $n\in \NN$. By  the proof of \cite{AKL}*{Lemma 8}\footnote{The same statement works when each $A_n$ is taken to be big and semi-ample. This will be used in Section~3.}, 
one has the  equality
\[
\Delta_{(C,x)}(D) \equ \bigcap_{n\in\NN}^{\infty}\Delta_{(C,x)}(D+A_n)\ .
\]
This reduces  the proof to the $\QQ$-divisor case via  the continuity of Zariski decomposition (see \cite{BKS}*{Proposition 1.14}).
\end{proof}

\section{Newton--Okounkov polygons and special loci associated to divisors}

\subsection{Local constancy of Newton--Okounkov polygons}
The goal of this subsection is to prove that the set of all Newton--Okounkov polygons, where the flags are taken to be centered at a very generic point $x$, is independent of the point. 

Let $X$ be a smooth projective surface, $x\in X$ a point and 
$D$ a big $\RR$-divisor on $X$. Denote by
\[
\mathcal{K}(D,x) \equ \{ \Delta\subseteq\RR^2_+ \ | \ \exists (C,x) \textup{ an admissible flag}, \textup{ such that  } \Delta =\Delta_{(C,x)}(D)\}\ ,
\]
the set of all Newton--Okounkov polygons of $D$, where the flags are based at $x$. By \cite{KLM}, this set is countable. 
Our  goal is to figure out how this set varies with the point $x$. 

\begin{theorem}\label{thm:main0}
With notation as above, there exists a subset $F=\cup_{m\in\NN}F_m\subseteq X$ consisting of a countable union of Zariski-closed proper subsets
$F_m\subsetneq X$ such that the set $\mathcal{K}(D,x)$ is independent of $x\in X\setminus F$.
\end{theorem}

The proof relies on the following observation. 

\begin{lemma}\label{prop:local constancy}
 Let $D$ be a big $\RR$-divisor on $X$,  $U\subseteq X$ be a subset with the following properties
 \begin{enumerate}
  \item $U$ is disjoint from all negative curves on $X$,
  \item for every $x,x'\in U$ and $(C,x)$ an admissible flag on $X$, there exists an irreducible curve $C'$ such that 
   $(C',x')$ is again admissible, and $C'\equiv C$.
 \end{enumerate}
 Then $\shk (D,x)$ is independent of $x\in U$. 
\end{lemma}

\begin{proof}
Let $x,x'\in U$, and $(C,x)$ as in the statement. Fix an admissible flag $(C',x')$, such that $C'\equiv C$.
Observe that $\nu_{(C,x)}(D)=\nu_{(C',x')}(D)=0$, as $U$ is disjoint from all the negative curves on $X$. Since $C\equiv C'$, we have 
$D-tC \equiv D-tC'$, and therefore $\mu_{(C,x)}(D)=\mu_{(C',x')}(D)$ as well. 

Again, as $U$ avoids all negative curves, $\alpha_{(C,x)}(t) =  \alpha_{(C',x')}(t) = 0$
for all $0\leq t\leq \mu(D)$.  Finally, since Zariski decomposition respects numerical equivalence, 
$P_{D-tC}\equiv P_{D-tC'}$ for all $0\leq t\leq \mu_D(t)$, therefore
$\beta_{(C,x)}(t) \equ \beta_{(C',x')}(t)$ for all $0\leq t\leq \mu(D)$,
hence $\Delta_{(C,x)}(D) \equ \Delta_{(C',x')}(D)$,  as required. 
\end{proof}

\begin{remark}\label{rmk:normal cycles} 
Roughly speaking the main idea of the proof of Theorem~\ref{thm:main0} is that whenever  $x\in X$ is a very general point then any cycle passing 
through this point can be deformed  non-trivially in its numerical equivalence class. The source for this material is  \cite{K}*{Chapter 2}.

To be more precise, \cite{K}*{Proposition 2.5} says that there exists a countable union of proper closed subvarieties $F\subseteq X$ such that for any $x\in VG(X)\deq X\setminus F$ and any birational morphism $w_0:C\rightarrow X$, where $C$ is a smooth irreducible curve with 
$x\in \textup{Im}(w_0)$, there exists a topologically trivial family of normal cycles
\begin{equation}\label{eq:cycle}
\xymatrix{  
U \ar[r]_{u} \ar[d]^{p} & X \\
T & }
\end{equation}
where $T$ is irreducible, $p$ is smooth (any fiber is smooth irreducible curve), $u$ is dominant, 
for any $t\in T$ the morphism $u|_{U_t}:U_t=p^{-1}(t)\rightarrow X$ is a birational morphism, and there exists $t_0\in T$, 
such that $u|_{U_{t_0}}:U_{t_0}\rightarrow X$ is the map $w_0$. 
\end{remark}

\begin{proof}[Proof of Theorem~\ref{thm:main0}]
First let us make the following definition. For every element $\gamma\in \textup{N}^1(X)$, the N\'eron--Severi group, let 
\[
 \Sing (\gamma) \deq \st{x\in X\mid x\in\Sing(C)\ \text{for every curve $C\in\gamma$ with $x\in X$}}\ .
\]
Note that $\Sing (\gamma)$ is always a proper subvariety of $X$. Let $VG(X)\subseteq X$ be the very general 
subset from Remark~\ref{rmk:normal cycles}; 
for every $\gamma\in \textup{N}^1(X)$, fix a topologically trivial family of normal cycles $u_\gamma$ with fibre class $\gamma$ 
(as in Remark~\ref{rmk:normal cycles}), and set 
\[
 \Defect (u_\gamma) \deq X\setminus \textup{Im}(u_\gamma) \ . 
\]
Since $u_\gamma$ is dominant, $\Defect(u_\gamma)$ is contained in a proper closed subvariety of $X$. Let
\[
 \shs \deq \bigcup_{\gamma\in \textup{N}^1(X)} \Sing(\gamma)\ \ \text{ and }\ \ \shd \deq \bigcup_{\gamma\in \textup{N}^1(X)} \Defect(u_\gamma)\ .
\]
We will apply
Lemma~\ref{prop:local constancy} with 
\[
\shu \deq VG(X) \setminus (\shs\cup\shd\cup ( \bigcup_{C \text{ is a negative curve}} C ) )\ .
\]
Observe that by construction $\shu$ is  disjoint from  all negative curves on $X$. 

Let  $x,x'\in \shu$, and $(C,x)$ be  an admissible flag. Then  $x\notin \Defect(u_\gamma)$, hence $u_\gamma$ has a fibre through $x$, whose image is $C$, and the same applies to $x'$, let us call this curve $C'$. By construction $C'$ is  irreducible. 
Although $C'$ might itself be singular at $x'$, there will be a curve numerically
equivalent to it which is smooth at $x'$, as $x'\notin \Sing([C])$ by the construction of $U$. 
\end{proof}

\subsection{Null and negative loci versus Newton--Okounkov polygons}

As mentioned above, the complements of the null or the negative loci describe 
the set of  points on $X$ where $D$ is positive locally. The following theorem explains this philosophy in the language of Newton--Okounkov polygons.

\begin{theorem}\label{thm:main1}
Let $X$ be a  smooth projective surface, $D$ be a big $\RR$-divisor on $X$ and $x\in X$ be an arbitrary point. Then 
\begin{enumerate}[i]
\item[(i)]  $x\notin \Neg(D)$ if and only if there exists an admissible flag $(C,x)$ such that the Newton--Okounkov polygon 
$\Delta_{(C,x)}(D)$ contains the origin $(0,0)\in \RR^2$.
\item[(ii)]  $x\notin \Null(D)$ if and only if there exists an admissible flag $(C,x)$ and a positive real number $\lambda >0$ such that 
$\Delta_{\lambda}\subseteq \Delta_{(C,x)}(D)$.
\end{enumerate}
\end{theorem}

\begin{remark}
It is immediate to see that $D$ is nef if and only if $\Neg(D)=\varnothing$, while $D$ is ample if and only if $\Null(D)=\varnothing$. Thus the corresponding nefness and ampleness criteria of Theorem A follow immediately from Theorem~\ref{thm:main1}.
\end{remark}

\begin{remark}
Theorem~\ref{thm:main1} implies that whenever  $(0,0)\in \Delta_{(C,x)}(D)$ for some admissible flag $(C,x)$, then the same holds  for all admissible flags centered at $x$. The analogous statement about points not contained in $\textup{Null}(D)$ is true as well. 
\end{remark}

\begin{example}\label{ex:interesting}
We discuss an  example of a  Newton--Okounkov polygon that  does not contain vertical/horizontal edges emanating from the origin, but the origin is contained in it. Let $X=\textup{Bl}_P(\PP^2)$ be the blow-up of the projective plane $\PP^2$ at the point $P$, $H$  the pull-back of $\sO_{\PP^2}(1)$, $E$ the exceptional curve, and  $x\in E$ be a point. Furthermore, let $\pi:X'\rightarrow X$ be the blow-up of $X$ at a point $x\in E$. Denote by $E_1$ the exceptional divisor of $\pi$, by $E_2$ the proper transform of $E$ on $X'$, and $E_3$  the proper transform of the line in $\PP^2$ passing through $P$ with the tangent direction given by the point $x\in E$. 

It is not hard to see that $(E_1^2) \equ (E_3^2) \equ -1$ and $(E_2^2) \equ -2$. Notice in addition that $E_1$ intersects both $E_2$ and $E_3$ at different points transversally, while the pair $E_2$ and $E_3$ does not intersect. The classes $E_1,E_2$, and $E_3$ generate the space $\textup{N}^1(X')_{\RR}$. A quick computation gives $\pi^*(H) \equ 2E_1+E_2+E_3$. Thus
\[ 
(\pi^*(D)-tE_1 \cdot E_2) \equ ((2-t)E_1+E_2+E_3 \cdot  E_2) \equ -t\ ,
\]
hence  $E_2\subseteq \Neg(\pi^*(D)-E_1)$ for all $0<t\ll 1$.

Let $\{y\} \deq E_1\cap E_2$. Comparing  with \cite{LM}*{Theorem 6.4}, note  that $\alpha(t)>0$ for any $0<t\ll 1$ with respect to the flag $(E_1,y)$. Furthermore, by Proposition~\ref{prop:propinf} and the above, one sees easily that the Newton--Okounkov polygon $\Delta_{(E_1,y)}(\pi^*(H))$ does not contain either a horizontal or a vertical edge starting at the origin, but contains the origin. This convex polygon is what we call in the next section the infinitesimal Newton--Okounkov polygon of the divisor $H$ on $X$ at the point $x$.

It is interesting to note that $X'$ has three negative curves $E_1,E_2$ and $E_3$, but the nef cone is minimally generated by four classes: 
$\pi^*(H),\pi^*(H)+E_3,\pi^*(H)+E_1+E_3$ and $E_1+E_3$
\end{example}

\begin{proof}[Proof of Theorem~\ref{thm:main1}]
$(i)$ By \cite{LM}*{Theorem 6.4} we obtain the following sequence of equivalences: 
\begin{eqnarray*}
(0,0)\in\Delta_{(C,x)}(D) & \Leftrightarrow & \nu\equ 0\ \ \text{ and }\ \ \alpha(0) \equ 0 \ , \\
& \Leftrightarrow & C \nsubseteq\Neg(D)\ \  \text{ and }\ \ x\notin \Neg(D) \\
& \Leftrightarrow & x\notin \Neg(D)\ ,
\end{eqnarray*}
which is what we wanted.

$(ii)$ ``$\Rightarrow$'' Let us  assume that $x\notin \Null(D)$. Since $\Neg(D)\subseteq \Null(D)$, this implies  $x\notin \Neg(D)$. 
By Lemma~\ref{lem:positive} we  can also assume without loss of generality that $D=P_D$, that is, $D$ is big and nef. These conditions yield  that $(D.C)>0$ for any irreducible curve $C\subseteq X$ passing through $x$. In particular the convex polygon $\Delta_{(C,x)}(D)$  contains the origin $(0,0)$ and the vertical segment $\{0\}\times [0,(D.C)]$ with non-empty interior for any admissible flag $(C,x)$. 

By fixing an admissible flag $(C,x)$, it remains then to show that the polygon $\Delta_{(C,x)}(D)$ contains a horizontal segment with non-empty interior starting at the origin. On the other hand, by the  convexity  of  Newton--Okounkov polygons, statement $(i)$, and Remark~\ref{rem:shift}, in order to prove the latter condition, it is enough to show that there exists $t>0$, such that $x\notin \textup{Neg}(D-t C)$.

By \cite{BKS}*{Theorem~1.1}, the divisor $D$ has an open neighborhood $\mathscr{U}$ inside the big cone, which intersects only finitely many Zariski chambers. Thus, there exist finitely many curves $\Gamma_1,\dots,\Gamma_s, \Gamma_{s+1}',\ldots ,\Gamma_{r}'$ such that 
\[
\Neg(D-t C) \ \subseteq \ \big(\bigcup_{i=1}^{i=s}\Gamma_i\big)\ \cup\ \big(\bigcup_{j=s+1}^{j=r}\Gamma_j'\big), \textup{ whenever }D-t C\in \mathscr{U} \ ,
\]
where $\Gamma_i$ are the curves containing $x$ and $\Gamma_j'$ are those which don't contain this point. By the above $\gamma_i\deq (D\cdot\Gamma_i) >0$ and thus $\gamma\deq \min_{1\leq i\leq s}\gamma_i>0$. 

Observe  that by possibly shrinking $\mathscr{U}$, we can arrange that even its closure intersects only finitely many Zariski chambers and remains inside the big cone.  Let 
\[
 D-t C\equ P_{t} + \left( \sum_{i=1}^{s}a_i^{t}\Gamma_i  + \sum_{j=s+1}^{r}b_j^{t}\Gamma_j'\right)
\]
be the Zariski decomposition of $D-t C$. By \cite{KLM}*{Proposition 2.1}, the function $a_i^{t}$ and $b_{j}^{t}$ depend continuously on $t$ as long as $D-t C\in\mathscr{U}$. Also, $a_i^{0}=b_j^{0}=0$ for all $i$ and $j$, since $D$ is nef.

Consider now the divisor 
\[
 D'_{t} \deq D-t C-\sum_{j=s+1}^{r}b_j^{t}\Gamma_j' \equ P_{t} + \sum_{i=1}^{s}a_i^{t}\Gamma_i\ .
\]
Notice that the right-hand side is the Zariski decomposition of $D'_{t}$. For every $1\leq i\leq s$ one has  
\[
 (D'_{t}\cdot \Gamma_i) \ \equ \ (D\cdot \Gamma_i) - t(C\cdot\Gamma_i) - (\sum_{j=s+1}^{r}b^{t}_j\Gamma_j')\cdot \Gamma_i)  \ \geq \ \gamma  - t(C\cdot\Gamma_i)  - \sum_{j=s+1}^{r}b_j^{t}(\Gamma_j'\cdot\Gamma_i) \ .
\]
By taking $t$ sufficiently small and by the continuity of $b_{j}^{t}$ as a function of $t$, then $(D'_{t}\cdot \Gamma_i) > 0$ for all $1\leq i\leq s$. On the other hand these negative curves are the only candidates for components of $\Neg(D'_t)$. Consequently, all the $a^{t}_{i}$'s are 
zero, and 
\[
 D - t C \equ P_{t} + \sum_{j=s+1}^{r}b^{t}_j\Gamma_j'
\]
is the Zariski decomposition of $D-t C$. Since the curves $\Gamma_j'$ are exactly the ones not containing $x$, we arrive at the desired conclusion
$x\notin \Neg(D-t C)$ for small $t$. 

``$\Leftarrow$'' We consider first the case when $D$ is a big and nef divisor, i.e. $D=P_D$ and $N_D=0$. 
Suppose  for a contradiction that there exists a curve $E\subseteq\Null(D)$ such that $x\in E$. If $C=E$, then, by the description of the Newton--Okounkov polygon of $D$, $\beta(0)=0$, as $(D.E)=0$. Thus, 
\[
\Delta_{(C,x)}(D) \ \bigcap \ \big(\{0\}\times\RR\big) \ = \ (0,0)\ .
\]
Consequently, $\Delta_{(C,x)}(D)$ cannot contain a small standard simplex. 

If $C\neq E$, then $C\cdot E>0$, as both $C$ and $E$ contain the point $x$. 
Thus, $(D-tC)\cdot E<0$  for any $t\ll 1$, This implies that  $E\subseteq \Supp (N_t)$ 
and consequently that $\alpha(t)=\textup{ord}_x(N_t|_C)>0$ for $t\ll 1$. Therefore,
\[
\Delta_{(C,x)}(D) \ \bigcap \ \big(\RR\times\{0\}\big) \ = \ (0,0)\ ,
\]
and again $\Delta_{(C,x)}(D)$ does not contain a standard simplex of any size. This leads to a contradiction to the existence of the curve $E$.

The general case, when $D$ is big, follows immediately from Lemma~\ref{lem:positive} and the observation that the condition $x\notin \Neg(D)$ is implied by the equivalence in statement $(i)$.
\end{proof}

\subsection{Valuative points.} By the definition given by Lazarsfeld and Musta{\c{t}}{\u{a}}, the
Newton--Okounkov polygon of a big $\QQ$-divisor $D$ encodes how all the sections of all powers of $D$ vanish along a fixed flag. Although not observed in \cite{LM}, the points that come from evaluating sections form a dense subset in the Newton--Okounkov polygon (hence, a posteriori there is no need for forming the convex hull in the construction). 

Conversely, it is a very exciting problem to find out exactly which points in the plane are given by valuations of sections, whether these points  lie in the interior or the boundary of the Newton--Okounkov polygon. 
To provide a partial answer,  we start with the following definition:
\begin{definition}
Let $D$ be a  big $\QQ$ ($\RR$)-divisor and $(C,x)$ an admissible flag on $X$. We say that a point $(t,y)\in \Delta_{(C,x)}(D)\cap\QQ^2$ 
(or $(t,y)\in\Delta_{(C,x)}(D)\cap\RR^2$ in the case of real divisors) is \emph{a valuative point of} $D$ \emph{with respect to the flag} $(C,x)$, 
if there exists an effective $\QQ$-divisor $D'\sim_{\QQ}D$ (an  effective $\RR$-divisor $D'\equiv D$) satisfying the property 
$\nu_{(C,x)}(D')=(t,y)$.  
\end{definition}

\begin{remark}
The fact that certain rational points in a Newton--Okounkov polygon are valuative is equivalent to the existence of sections with prescribed 
vanishing behaviour along the given flag in a linear series $|mD|$ for $m\gg 0$.  
\end{remark}

\begin{corollary}\label{cor:valpoints}
Let $D$ be a big $\QQ$-divisor and $(C,x)$ be an admissible flag on $X$. Then 
\begin{itemize}
\item[(i)]  Any rational point in $\textup{int}(\Delta_{(C,x)}(D))$ is a valuative point;
\item[(ii)] Suppose $\Delta_{\lambda,\lambda'}\subseteq \Delta_{(C,x)}(D)$ for some $\lambda,\lambda'>0$. Then any rational point on the horizontal segment $[0,\lambda)\times\{0\}$ and the vertical one $\{0\}\times [0,\lambda')$ is a valuative point.
\end{itemize}
\end{corollary}

\begin{remark}\label{rem:real}
If instead we consider $D$ to be a big $\RR$-divisor and change in the definition of valuative points $\QQ$-linear equivalence by numerical equivalence, we obtain the same statement as in Corollary~\ref{cor:valpoints} in this more general setup.
\end{remark}

\begin{remark}
For the vertical segment $\{0\}\times[0,\lambda')$, one can obtain the statement as a consequence of  \cite{ELMNP2}*{Theorem 2.13}, as illustrated  in \cite{ELMNP2}*{Example 2.13}, along  with the restriction theorem \cite{LM}*{Theorem 4.24} for  Newton--Okounkov bodies. Here we give a different proof for the surface case, relying only on ideas of convex geometric nature arising from the theory developed so far.
\end{remark}

\begin{remark}\label{rem:vanishing}
Let $A$ be an ample $\QQ$-divisor on $X$, $C\subseteq X$ be a rational curve and $x\in C$ a smooth point, set $d = (A\cdot C)$. Then $(0,d)$, the highest vertex of the polygon $\Delta_{(C,x)}(A)$ on the $y$-axis, is also  a valuative point. The argument follows from  Serre vanishing and the rationality of $C$. 

Since $A$ is ample,  $H ^1(X,\sO_{X}(mA-C))=0$ for all $m\gg 0$ by Serre vanishing. Therefore, the restriction maps $H^0(X,\sO_{X}(mA))\longrightarrow H^0(\PP^1,\sO_{\PP^1}(md))$
are  all surjective. As  $C$ is a rational curve, there exists $\overline{s}\in H^0(\PP^1,\sO_{\PP^1}(md))$ such that $\mult_{x}(\overline{s})=md$. By the surjectivity of the restriction maps there exists a section $s\in H^0(X,\sO_{X}(mA))$ whose image is $\overline{s}$. In particular, $\nu_{(C,x)}(s)=(0,d)$.

The  proof shows that in fact all  rational points on the edge of $\Delta_{(C,x)}(A)$ with vertices at $(0,0)$ and $(0,d)$ are valuative points. In this sense  Corollary~\ref{cor:valpoints} serves as a local restriction theorem, and  Newton--Okounkov polygons give us some elbow room to obtain  local statements without having to rely on  vanishing theorems.

As explained in \cite{AKL}*{Proposition~14}, if $C$ is not rational, then there  exist line bundles $L$ of degree $d>0$ on $C$, so that no section $s\in H^0(C,mL)$ has $\ord_x(s)=md$ for any $m>0$. Thus, the rationality of $C$ is crucial. If $D$ is a big divisor and $C$ is a rational curve then one has to assume additionally that $\Null(D)\cap C=\varnothing$. Based on the proof of Corollary~\ref{cor:ratvaluation}, this will imply that the highest vertex of the polygon $\Delta_{(c,x)}(D)$ on the $y$-axis is also a valuative point. 
\end{remark}

\begin{proof}[Proof of Corollary~\ref{cor:valpoints}]
$(i)$ The following remark will be used repeatedly throughout this proof: given two valuative points $A=(t,y), B=(t',y')\in\Delta_{(C,x)}(D)$, any rational point contained in the line segment $[AB]$, connecting the point $A$ with $B$, is again a valuative point. This is due to the fact that  $\nu_{(C,x)}$ is a valuation map.

Let $(t_0,y_0)\in\textup{int}(\Delta_{(C,x)}(D)$ be a point with rational coordinates. The idea is to show that
there exist valuative points on the vertical line $t=t_0$ above and below $(t_0,y_0)$. We verify the existence of a valuative point lying above $(t_0,y_0)$, the other case being completely analogous. Consider the interior of the shape $\Delta_{(C,x)}(D)\cap\{(t,y)|y\geq y_0\}$, which is divided into two non-empty subsets by the line $t=t_0$. Since  valuative points are dense in each subset, we can choose a point in each of them. The line segment connecting these two points intersects the line $t=t_0$ in a rational point that is above $(t_0,y_0)$. Hence by the above observations this point is also valuative. 

$(ii)$ We check first the assertion on the horizontal line segment. Let $\delta\in [0,\lambda)$ be a rational number. By Remark~\ref{rem:shift}, the polygon $\Delta_{(C,x)}(D-\delta C)$ contains a non-zero simplex. By Theorem~\ref{thm:main1} this latter condition implies  that $x\notin \Null(D-\delta C)$. Since $\textup{\textbf{B}}(D-\delta C)\dsubseteq \Null(D-\delta C)$, then 
$x\notin \textup{\textbf{B}}(D-\delta C)$. Thus, the origin $(0,0)\in\Delta_{(C,x)}(D-\delta C)$ is a valuative point with respect to the divisor $D-\delta C$. Using Remark~\ref{rem:shift}, then the point $(\delta, 0)\in\Delta_{(C,x)}(D)$  is also valuative with respect now to $D$.

It remains to show that all rational points on the open line segment $\{0\}\times (0,\lambda')$ are  valuative as well. To this end, observe  by Remark~\ref{rem:shift} that 
\begin{equation}\label{eq:sum}
\Delta_{(C,x)}(D+\epsilon C)_{t\geq \epsilon} \equ  (\epsilon,0)+\Delta_{(C,x)}(D) \ \ \  \textup{ for any }\epsilon>0\ .
\end{equation}
So, if we show that $\vol_X(D+\epsilon C)>\vol_X(D)$ for some rational $\epsilon>0$, then there will exist a valuative point in the area $\Delta_{(C,x)}(D+\epsilon C)\cap (0,\epsilon)\times \RR$. On the other hand this implies that the open line segment $\{\epsilon\}\times (0,\lambda' )$ is inside $\Delta_{(C,x)}(D+\epsilon C)$. By statement $(i)$, any rational point on this line segment is valuative for the $\QQ$-divisor $D+\epsilon C$. Applying  (\ref{eq:sum}) again along with Lazarsfeld and Musta\c t\u a's definition of  Newton--Okounkov polygons  yields that the same can be said about all the rational points on the vertical segment $\{0\}\times(0,\lambda')$ in the polygon $\Delta_{(C,x)}(D)$.

It remains to prove that $\vol_X(D+\epsilon C)>\vol_X(D)$ for some $0<\epsilon \ll 1$. To this end, assume first that $D=P$ is big and nef. Since $\Delta_{(C,x)}(D)$ contains a standard simplex, then by Theorem~\ref{thm:main1} we know that $C\nsubseteq \Null(D)$. This latter condition implies that 
$(D\cdot C)>0$. Thus, the  Nakai--Moishezon criterion yields that $D+\epsilon C$ is also nef for $\epsilon\ll 1$ and we have the following inequality
\[
\vol_X(D+\epsilon C) \equ \frac{(D+\epsilon C)^2}{2} \ > \ \frac{D^2}{2} \equ \vol_X(D)\ ,
\]
whenever $\epsilon \ll 1$, settling the claim for  $D$  big and nef. For the general case, let $D=P+N$ and 
$D+\epsilon C=P_{\epsilon}+N_{\epsilon}$ be the respective Zariski decompositions of $D$ and $D+\epsilon C$. By the previous step we know that $P+\epsilon C$ is nef for all $0<\epsilon\ll 1$, hence  one can write $D+\epsilon C = (P+\epsilon C)+N$, and by the minimality of the Zariski decomposition we see  that $N-N_{\epsilon}$ is effective. Consequently,  $P_{\epsilon}-(P+\epsilon C)$ is also effective. Using  the ideas from the big and nef case  we deduce that 
\[
\textup{vol}_X(D+\epsilon C)=\textup{vol}_X(P_{\epsilon})\ \geq \ \textup{vol}_X(P+\epsilon C)\ >\  \textup{vol}_X(P)=\textup{vol}_X(D) 
\]
for all $0<\epsilon \ll 1$. This also proves the corollary in the big case.
\end{proof}


\section{Moving Seshadri constants and infinitesimal Newton--Okounkov polygons}
 The goal that we pursue in this section is to study the relationship between the positivity properties of a big divisor and the geometry of the Newton--Okounkov polygons that can be defined on the blow-up of a point. We show how  moving Seshadri constants can be read off from these polygons,
 and study which of their boundary points  are valuative.

As before, we assume $X$ to be a smooth projective surface, $D$ a big $\QQ$ (or $\RR$) divisor on $X$, and $x\in X$ a point. We denote by $\pi:X'\rightarrow X$ the blow-up of $X$ at $x$ with  $E$  the exceptional divisor. For any point $y\in E$, we call the polygon $\Delta_{(E,y)}(\pi^*(D))$ the \emph{the infinitesimal Newton--Okounkov polygon} of $D$ attached 
to the admissible flag $(E,y)$. This concept originates in \cite{LM}*{Section~5} (note the deviation from the terminology of \cite{LM}). 
The goal of this section is to explore the relationship between the positivity properties of the divisor $D$ and the geometry of the polygons $\Delta_{(E,y)}(\pi^*(D))$.

\subsection{Infinitesimal Newton--Okounkov polygons} 
In this subsection, we study the basic properties of the infinitesimal Newton--Okounkov polygons.
For a big $\RR$-divisor $D$, write 
\[
\mu' \equ \mu'(D,x) \equ \mu (\pi^*(D), E) \ \deq   \ \textup{sup}\{t >0\ | \ \pi^*(D)- tE \textup{ is big }\} \ .
\]
It follows from \cite{LM}*{Theorem 6.4} and the definition of $\mu'$ that $\Delta_{(E,y)}(\pi^*(D))\,\subseteq\, \RR_+\times [0,\mu']$.

Furthermore, for all $x\in X$,  we associate to the divisor $D$ the set of all infinitesimal Newton--Okounkov polygons rooted at $x$:
\[
\mathcal{K}'(D,x) \deq  \{\Delta\subseteq \RR^2_+ \mid \exists y\in E, \textup{ such that } \Delta=\Delta_{(E,y)}(\pi^*(D))\}\ .
\]

\begin{proposition}\label{prop:propinf}
With notation as above, we have 
\begin{enumerate}
\item[(i)] $\Delta_{(E,y)}(\pi^*(D))\subseteq \Delta_{\mu'(D,x)}^{-1}$ for any $y\in E$;
\item[(ii)] there exist finitely many points $y_1,\ldots ,y_k\in E$ such that the polygon $\Delta_{(E,y)}(\pi^*(D))$ is independent of 
$y\in E\setminus\{y_1,\ldots,y_k\}$, with  base  the whole line segment $[0,\mu']\times\{0\}$. 
\end{enumerate}
\end{proposition}
\begin{remark}
The constant $\mu'(D';x)$ can be computed on $X$. If $|V|$ is a linear series on $X$, define 
\[
\overline{\mult}_x(|V|) \deq  \sup \{ \mult_x(F) | F\in |V|\}\ .
\]
If $D$ is a big Cartier divisor, set 
\[
\overline{\mult}_x(||D||) \deq  \limsup_{m\rightarrow \infty}\frac{\overline{\mult}_x(|pD|)}{p} \ .
\]
Then $\mu' \equ \overline{\mult}_x(||D||)$, by simple properties of the multiplicity (cf. \cite{DKMS}*{Proposition 3.2}). 
The same holds whenever $D$ is a big $\QQ$-divisor and by continuity the statement extends  to $\RR$-divisors.
\end{remark}
\begin{proof}
$(i)$ Based on the second part of the proof of Lemma~\ref{lem:positive}, it is not hard to see that it suffices  to show the statement when $D$ is merely a big Cartier divisor. It was pointed out above that $\Delta_{(E,y)}(\pi^*(D))\subseteq \RR\times [0,\mu'(D,x)]$. Thus, it remains to show that $\Delta_{(E,y)}(\pi^*(D))$ lies below the diagonal $y=t$. 

By Zariski's main theorem (see \cite{H}*{Theorem III.11.4}) one has the  isomorphisms
\[
H^0(X',\sO_{X'}(m\cdot\pi^*(D))) \ \simeq \ H^0(X,\sO_X(mD)) \ \ \text{for all $m>0$.}
\]
Hence all the sections of $\pi^*(D)$ can be seen as  pull-backs of sections from $X$. Let $D'\in |mD|$ for some $m>0$. In order to end the proof 
it is sufficient to check  the  inequality
\[
\mult_x(D') \equ \ord_E(\pi^*(D')) \ \geq \ \ord_y\big(\big(\pi^*(D')-\mult_x(D')E\big)|_E\big)\ .
\]
The main reason for this relation is that the multiplicity of a tangent direction of a given curve at $x$ cannot exceed the multiplicity at this point. This can in fact be checked locally: let $\{u_1,u_2\}$ be a local system of parameters in a neighborhood $U\subseteq X$ of the point $x$. Then a section $s\in H^0(X,\sO_X(mD))$ restricted to $U$  can be  written in terms of the local 
coordinates $u_1$ and $u_2$ as 
\[
s|_U \equ f_d(u_1,u_2)+f_{d+1}(u_1,u_2) \ + \ \ldots \ + \ f_{d+l}(u_1,u_2)\ ,
\]
where each $f_i$ is a homogeneous polynomial of degree $i$ with $d=\mult_x(s)$. Since we are working over the complex numbers and the polynomial $f_d$ is homogeneous, we can write it as follows 
\[
f_d(u_1,u_2) \equ  (u_1-\alpha_1u_2)^{i_1}\cdot (u_1-\alpha_2u_2)^{i_2}\cdot \ldots \cdot (u_1-\alpha_ku_2)^{i_k}
\]
where $i_j\in\NN$, $\alpha_i\in \CC$, and $\sum_{j=1}^{j=k}i_j=d$. 

Now, let 
\[
\pi|_U \colon U'\deq\{u_1v_1=u_2v_2\}\subseteq U\times \PP^1\ \longrightarrow  \ U
\]
be the blow-up of $U$ at $x$. The form of the decomposition of $f_d$ implies that it is enough to do the computations on the open subset
$U_1'=\{v_1=1\}\subseteq U'$. Then 
\[
\pi^*(s)|_{U_1'} \equ  u_2^d\cdot \Big(f_d(v_2,1)\ + \ u_2f_{d+1}(v_2,1)\ + \ \ldots \ + \ u_2^lf_{d+l}(v_2,1)\Big) \deq u_2^d\cdot F(u_2,v_2)\ .
\]
The first term of the right-hand side yields $\ord_E(\pi^*(s))=d$. The shape of the second one gives us 
\begin{displaymath}
\nu_2(\pi^*(s)) = \textup{ord}_{[1:\alpha]} (F(0,v_2)) \equ  \left\{ \begin{array}{ll}
i_j & \textrm{if $\alpha=\alpha_j$, for some  $i=1,\ldots ,l$}\ ,\\
0 & \textrm{otherwise} \ .
\end{array} \right.
\end{displaymath}
Since $d\geq i_j$ for all $1\leq j\leq k$ by construction, we are done.

$(ii)$ Let
\[
D'_t \deq  \pi^*(D)-tE =  P_t'+N_t'
\]
be the appropriate Zariski decomposition. Neither of the coefficient $\nu'$ of $E$ in the negative part $N_0'$ and $\mu'(D,x)$  depend on the choice of  $y\in E$, hence $\Delta_{(E,y)}(\pi^*(D))\subseteq [\nu',\mu']\times \RR_+$ for any $y\in E$. 

By Remark~\ref{rem:vertical}, for any $\xi\in [\nu',\mu']$ the length of the vertical slice $\Delta_{(E,y)}(\pi^*(D))_{t=\xi}$ 
is independent of  $y\in E$. Thus, to finish up, it suffices  to show that there exists finitely many points $y_1,\ldots ,y_k\in E$ 
such that $[\nu',\mu']\times\{0\}\subseteq \Delta_{(E,y)}(\pi^*(D))$ for all $y\in E\setminus\{y_1,\ldots ,y_k\}$. 

This, however,  is a consequence of \cite{KLM}*{Proposition 2.1} which states that the function $t\in [\nu',\mu'] \longrightarrow N_t'$ is increasing,
i.e. if $\nu'\leq t_1\leq t_2\leq \mu'$ then the divisor $N_{t_2}'-N_{t_1}'$ is effective. In particular, the divisor $N_{\mu'}'-N_{t}'$ is effective 
for any $t\in [\nu',\mu']$. Consequently,  the function $\alpha$ is identically zero  on the whole interval $[\nu',\mu']$ 
whenever $x\notin\textup{Supp}(N_{\mu'}')\cap E$, as stated.
\end{proof}

It makes now good sense to  introduce the following definition:
\begin{definition}
With notation as above, for a big $\RR$-divisor $D$ on $X$, we call the polygon $\Delta_{(E,y)}(\pi^*(D))$,  where the point $y\in E$ is chosen 
to be general, \emph{the generic infinitesimal Newton--Okounkov polygon}\footnote{In \cite{LM} this was originally named the 
infinitesimal Okounkov body.} of $D$ at $x$. We denote this polygon by $\Delta(D,x)$.
\end{definition}

The set of polygons $\mathcal{K}'(D,x)$ as we have seen above is finite. Furthermore, by \cite{LM}*{Theorem A}, we also know that all 
the polygons in this set have the same area  equal to $\textup{vol}_X(D)$. Whence  it is natural to ask what other data remains invariant 
for all  polygons in the finite set $\mathcal{K}'(D,x)$, thus giving rise to  natural invariants of $D$ and the point $x$. 

\begin{proposition}\label{prop:t-axis}
The set of all  $t$-coordinates of the vertices of the infinitesimal Newton--Okounkov polygon $\Delta_{(E,y)}(\pi^*(D))$ does not depend on $y$.
\end{proposition}
\begin{remark}
As we shall see in Theorem~\ref{thm:triangle} and Theorem~\ref{thm:main2} below, whenever $x\notin \Null(D)$, then the origin $(0,0)$, $(\epsilon,\epsilon)$ and $(\epsilon,0)$ are all vertices of the polygon $\Delta_{(E,y)}(\pi^*(D))$ for any $y\in E$, where $\epsilon = \epsilon(||D||,x)$ is the moving Seshadri constant.
\end{remark}
\begin{proof}
The main idea for the proof is that the function $\alpha$, defining the lower bound of the Newton--Okounkov polygon, is increasing and concave-up and $\beta$, defining the upper bound, is concave-down. So, let $t_0=\nu'< t_1<\ldots < t_{k-1}< t_k=\mu'$ be the sequence of the $t$-coordinates of all the vertices of the generic infinitesimal Newton--Okounkov polygon $\Delta(D,x)$. By Proposition~\ref{prop:propinf}, these coordinates come from the vertices sitting on the upper boundary defined by the function $\beta$.

So, let $\Delta$ be another infinitesimal Newton--Okounkov polygon that is not equal to $\Delta(D,x)$. Suppose there is an intermediate point $t'\in (t_i,t_{i+1})$, for some $i=1,\ldots ,k$, which is the $t$-coordinate of some vertex on $\Delta$. Assume first that this vertex is on the lower bound of this polygon. By Remark~\ref{rem:vertical} we know that for any $t''\in[\nu',\mu']$ the length of the vertical segment $\Delta_{(E,y)}(\pi^*(D))_{t=t''}$ does not depend on $y$. Furthermore the upper bound of $\Delta(D,x)$ is a straight segment in a neighborhood of the line $t=t'$. These two facts force the upper bound of $\Delta$ in a neighborhood of the vertex defining the coordinate $t'$ to be concave up. This leads to a contradiction since the upper bound is actually concave down. Also, the vertex on $\Delta$ giving $t'$ cannot be either on the upper bound of this polygon. This is due to the fact that the trapezoid in $\Delta(D,x)\cap [t_i,t_{i+1}]\times\RR$ is tranformed into the shape $\Delta\cap [t_i,t_{i+1}]\times\RR$ by an affine transformation, thus inducing lines into lines. 

The same reasoning implies  that if $t_i$ is the $t$-coordinate of some vertex on $\Delta(D,x)$ then this coordinate is also the $t$-coordinate for some vertex of $\Delta$.
\end{proof}

\subsection{Local constancy of generic infinitesimal Newton--Okounkov polygons.}
In the previous subsection we have attached a generic infinitesimal Newton--Okounkov polygon $\Delta(D;x)$ to a big  $\RR$-divisor $D$ and a point $x\in X$. In light of Theorem~\ref{thm:main0} it is then natural  
to study how $\Delta(D;x)$ varies when the point $x\in X$ moves around.

\begin{theorem}\label{thm:locinf}
Let  $D$ be a big $\RR$-divisor on a smooth projective surface $X$.
Then there exists a  subset $F'=\cup_{m\in\NN}F_m'\subseteq X$ consisting of a countable union of Zariski-closed proper subsets 
$F_m'\subsetneq X$ such that the polygon $\Delta(D,x)\subseteq \RR^2$ is independent of $x\in X\setminus F'$.
\end{theorem}

\begin{remark}
Suppose $A$ is an ample Cartier divisor on $X$. Proposition~\ref{prop:genericinf} below says that the polygon $\Delta(A,x)$, as explained in Theorem~\ref{thm:locinf} does not depend on $x$ for very general choices, is contained in an area determined by the global Seshadri constant $\epsilon(A)=\sup \{\epsilon(A,x)|x\in X\}$.
\end{remark}

\begin{proof}
By an argument similar to the  second part of proof of Lemma~\ref{lem:positive}, we can assume without loss of generality that 
$D$ is a big Cartier divisor. 
 
Denote by $p_1,p_2:X\times X\rightarrow X$ the respective projections onto the first and the second factors, let 
$\Delta_{X}\subseteq X\times X$ be the diagonal. Write  $\pi:Y\deq \textup{Bl}_{\Delta_X}(X\times X)\rightarrow X\times X$ for  
the blow-up along the diagonal with exceptional divisor  $E_X\subseteq Y$, and projection morphisms $\pi_1,\pi_2:Y\rightarrow X$. 

We will study  the family $\pi_1:Y\rightarrow X$, which has the property that for $x\in X$ the fiber $\pi_1^{-1}(x)=\textup{Bl}_x(X)$ is 
the blow-up of $X$ at $x$. Let $\mathcal{D}=\pi_2^*(D)$ and notice that $\mathcal{D}|_{\pi^{-1}_1(x)}=\pi_x^*(D)$, 
where $\pi_x=\pi|_{\pi^{-1}(x\times X)}:\textup{Bl}_x(X)\rightarrow X$. Consider  the incomplete flag 
\[
Y_0=Y\supseteq Y_1=E_X\supseteq Y_2 \ ,
\]
where $Y_2$ is defined as follows: because the diagonal in $X\times X$ is smooth, $E_X$ is a projective bundle over $\Delta_X$; now let  $Y_2$
be an arbitrary section of $E_X\to \Delta_X$. Then $Y_2$ is a section of  $\pi_1$. Denote by $E_x\deq E_X\cap \pi^{-1}_1(x)$, the exceptional divisor of the map $\pi_x$, and by $y_{x}=Y_2\cap\pi^{-1}_1(x)\in E_x$.  

Thus the goal is to understand  the family of Newton--Okounkov polygons $\Delta_{(E_x,y_{x})}(\pi^*_x(D))$ for $x\in X$. Applying 
\cite{LM}*{Theorem 5.1} to the flat family $\pi_1:Y\rightarrow X$, one deduces that there exists a countable family $F'=\cup_{m\in\NN}F'_m\subseteq X$, where each $F'_m\subseteq X$ is a proper Zariski closed subvariety, satisfying the property that 
\[
\Delta_{(E_x,y_{x})}(\pi_x^*(D)) \ \subseteq \ \RR^2 \text{ is independent of $x$ for $x\in X\setminus F'$.}
\]
Let $[\nu',\mu']$ be the support on the $t$-axis of the polygon $\Delta\deq \Delta_{(E_x,y_{x})}(\pi_x^*(D))$ for $x\in X\setminus F'$. 
If $[\nu',\mu']\times\{0\}\subseteq \Delta$, then, by Proposition~\ref{prop:propinf}, the set $\Delta_{(E_x,y_{x})}(\pi_x^*(D))$ is the 
generic infinitesimal Newton--Okounkov polygon of $D$ at $x$ for any $x\in X\setminus F'$. 

Otherwise, by Proposition~\ref{prop:t-axis}, we know that the $t$-coordinates of the vertices of $\Delta$ are the $t$-coordinates of the generic infinitesimal Newton--Okounkov polygon $\Delta(D,x)$ for any $x\in X\setminus F'$. 

Furthermore, by Remark~\ref{rem:vertical} for any $t\in [\nu',\mu']$ the length of the vertical segment $\Delta\cap\{t\}\times\RR$ is equal to the length of the vertical segment $\Delta(D,x)$ for any $x\in X\setminus F'$. These two latter facts imply that $\Delta(D,x)\subseteq \RR^2$ is independent  of $x\in X\setminus F'$ which concludes the proof.
\end{proof}

\subsection{Infinitesimal Newton--Okounkov polygons and base loci.}
In this subsection, we are trying to explore how certain ideas of Nakamaye (see \cite{Nak1}), connecting augmented base loci and blow-ups can be seen in the language of infinitesimal Newton--Okounkov polygons. 

Returning  to Example~\ref{ex:interesting}, we remark that the infinitesimal Newton--Okounkov considered there does not contain a triangle of the form the $\Delta_{\xi}^{-1}$ for some $\xi>0$. Notice also that the base point was taken to be contained in the null locus. These observations  lead to our first goal, namely, to find  conditions under which all infinitesimal Newton--Okounkov polygons contain a triangle $\Delta_{\xi}^{-1}$ for some $\xi>0$. We shall see below that this information suffices to describe the complement of the null locus. 

Furthermore, we discuss here how the points of the negative locus  can be read from infinitesimal data. 
This kind of connection has not been looked at before and completes the picture that started in \cite{Nak1} in a clean  way. 

\begin{theorem}\label{thm:triangle}
Let  $D$ be  a big $\RR$-divisor on a smooth projective surface $X$. Then 
\begin{enumerate}
\item $x\notin \Neg(D)$ if and only if $(0,0)\in\Delta_{(E,y)}(\pi^*(D))$ for any $y\in E$,
\item $x\notin \Null(D)$ if and only if there exists $\xi>0$ such that $\Delta_{\xi}^{-1}\subseteq\Delta_{(E,y)}(\pi^*(D))$ for any $y\in E$.
\end{enumerate}
\end{theorem}
\begin{remark}
Theorem~\ref{thm:triangle} seems to imply that in order to check whether  a point is contained in the negative or the null locus one needs to know all  infinitesimal Newton--Okounkov polygons at the point. Analogously to Theorem~\ref{thm:main1}, we prove in Lemma~\ref{lemma:function constant} that it suffices to check  the condition above for just one point $y\in E$.
\end{remark}

\begin{proof}
$(1)$ Based on the ideas from the second part of the proof of Lemma~\ref{lem:positive}, it is not hard to see that it suffices 
to check  the statement in  the case when $D$ is a big $\QQ$-divisor. 

Assume first  that $x\notin\Neg(D)$, and let $D=P_D+N_D$ be the corresponding Zariski decomposition. Then by Remark~\ref{rem:ZD of pullback} 
implies that $\pi^*(D)=\pi^*(P_D)+\pi^*(N_D)$ is the Zariski decomposition of $\pi^*D$. In particular, $\Neg(\pi^*(D))\cap E=\varnothing$, and 
Theorem~\ref{thm:main1} yields   $(0,0)\in\Delta_{(E,y)}(\pi^*(D))$.

For the reverse  implication suppose on the contrary  that $x\in\Neg(D)$. By scaling we can assume that $D$, $P_D$, and $N_D$ are all integral. 
Let $x\in C\subseteq X$ be an irreducible curve appearing in $N_D$ with a strictly positive coefficient $a>0$. By  \cite{PAG}*{Proposition 2.3.21}
we know that for any natural number $m>0$ and any effective divisor $D'\in |mD|$ there exists another effective divisor $P'\in |mP_D|$ for which 
$D'=P'+N_D$. Therefore, $\mult_{C}(D')\geq m\cdot a$. 

On the other hand, as the polygon $\Delta_{(E,y)}(\pi^*(D))$ is the closure of all  normalized valuation vectors  of the effective divisors $\pi^*(D')$, 
for any $D'\in|mD|$ and any $m>0$, we obtain by the above that  
\[
\nu_1(\pi^*(D')) \equ \mult_E(\pi^*(D'))\geq ma\cdot \mult_x(C)\ .
\]
Consequently, by the  definition of  the Newton--Okounkov polygons we obtain that  $(0,0)\notin\Delta_{(E,y)}(\pi^*(D))$,  contradicting
our initial assumption.

$(2)$ We start with the direct implication. Since  $x\notin\Null(D)$, we know    via Theorem~\ref{thm:main1}  that  the polygon $\Delta_{(C,x)}(D)$ contains a small standard simplex  for any flag $(C,x)$.
 
Let $C_1,C_2\subseteq X$ be two irreducible curves intersecting transversally at $x$ (in particular $x$ is a smooth point of both). Then there exists $\lambda>0$, such that $\Delta_{\lambda}$ is contained both in $\Delta_{(C_1,x)}(D)$ and $\Delta_{(C_2,x)}(D)$. Now, for any given real number $0<\xi<\lambda$, the point $(\xi,0)$ is a valuative point of $D$ with respect to both flags $(C_1,x)$ and $(C_2,x)$ according to Corollary~\ref{cor:valpoints}, as pointed out in Remark~\ref{rem:real}. In particular, there exist effective $\RR$-divisors $D_1$ and $D_2$, both numerically equivalent to $D$, such that $D_i = \xi C_i+D_i'$,  where $x\notin \textup{Supp}(D_i')$, for any $i=1,2$. Since each $C_i$ is smooth at $x$, we have $\pi^*(C_i) \equ \overline{C}_i+E$, where $\overline{C}_i$ is the proper transform of $C_i$. Set $y_i=\overline{C}_i\cap E$; as  $C_1$ and $C_2$ intersect transversally at $x$, 
one has $y_1\neq y_2$. Note also that each $\pi^*(D_i)$ contributes to the Newton-=Okounkov polygon $\Delta_{(E,y)}(\pi^*(D))$.
Therefore,  Remark~\ref{rem:definition} yields that 
\begin{displaymath}
\nu_{(E,y)}(\pi^*(D_i)) = \left\{ \begin{array}{ll}
(\xi, 0) & \textrm{if $y\neq y_i$}\\
(\xi,\xi) & \textrm{if $y=y_i$}
\end{array} \right.
\end{displaymath}
for any $i=1,2$. Since $y_1\neq y_2$, we obtain  that $\Delta_{\xi}^{-1}\subseteq \Delta_{(E,y)}(\pi^*(D))$ for any $y\in E$.

For the reverse  implication, assume that there exists $\xi>0$ with $\Delta_{\xi}^{-1}\subseteq \Delta_{(E,y)}(\pi^*(D))$ for any $y\in E$. By Remark~\ref{rem:shift},  the infinitesimal polygon $\Delta_{(E,y)}\big(\pi^*(D)-tE\big)$ contains a small simplex for any real number $0<t<\xi$ and any $y\in E$. As a consequence, Theorem~\ref{thm:main1} gives that 
\begin{equation}\label{eq:null1}
\Null\big(\pi^*(D)-tE\big) \ \bigcap \ E \equ \varnothing 
\end{equation}
for any rational $0<t \ll 1$.

We intend to reduce the problem to the case when $D=P_D$ is big and nef. As $(0,0)\in\Delta_{(E,y)}(\pi^*(D))$, then $x\notin\Neg(D)$. Thus, by Remark~\ref{rem:ZD of pullback}, we know that $\pi^*(D)=\pi^*(P_D)+\pi^*(N_D)$ is the Zariski decomposition of $\pi^*(D)$. This in turn implies that  $\Supp(\pi^*(N_D))\cap E=\varnothing$. Now, by Lemma~\ref{lem:positive}, this yields that 
\[
\Delta_{(E,y)}(\pi^*(D)) \equ \Delta_{(E,y)}(\pi^*(P_D))
\]
for any $y\in E$, which allows us to reduce the problem to the big and nef case.

Assume now that $D=P_D$ is big and nef; aiming at a contradiction suppose  that there exists an irreducible curve $C\subseteq \Null(D)$ containing $x$. This implies that $(D\cdot C)=0$. If we denote by $\overline{C} \deq \pi^*(C)-\mult_x(C)E$ the proper transform of $C$, then for all $0<t\ll 1$ one has
\[
\big((\pi^*(D)-tE)\cdot \overline{C}\big) \equ  (D\cdot C)-t\cdot \textup{mult}_x(C) \ < \ 0 \ .
\] 
The algorithm for finding  Zariski decompositions then yields that $\overline{C}\dsubseteq \Neg\big(\pi^*(D)-tE\big)$, hence $\overline{C}\dsubseteq\Null(\pi^*(D)-tE)$. Since $\overline{C}\cap E\neq \varnothing$, this contradicts $(\ref{eq:null1})$, and we are done. 
\end{proof}

An interesting consequence of the above statement is the following criterion for a point not to be in the null locus of a big real divisor. 

\begin{corollary}
In the setting of Theorem~\ref{thm:triangle}, one has  $x\notin\Null(D)$ if and only if there exist  irreducible curves $C_1, C_2\subseteq X$ 
that intersect transversally at $x$, and a positive number $\lambda>0$ such that the horizontal segment $[0,\lambda]\times\{0\}$ 
is contained both in $\Delta_{(C_1,x)}(D)$ and $\Delta_{(C_2,x)}(D)$.
\end{corollary}
\begin{proof}
Choose a positive real number $\lambda'\in (0,\lambda)$. By Remark~\ref{rem:shift} each $\Delta_{(C_i,x)}\big(D-\lambda'C_i\big)$ contains a small simplex. By Corollary~\ref{cor:valpoints} the origin $(0,0)$ is a valuative point in each of these polygons. Applying  Remark~\ref{rem:shift} again,  we see that the point $(\lambda',0)$ is a valuative point in each $\Delta_{(C_i,x)}(D)$. Using this fact and the last part of the proof of the direct implication of Theorem~\ref{thm:triangle}, one deduces that the polygon $\Delta_{(E,y)}(\pi^*(D))$ contains the triangle $\Delta_{\lambda'}^{-1}$ for any $y\in E$. By Theorem~\ref{thm:triangle}, this implies that $y\notin\Null(D)$. 

The reverse implication is an easy consequence of Theorem~\ref{thm:main1}.
\end{proof}

\subsection{Moving Seshadri constants.}
It has been long known  (and has been illustrated in the previous section in connection with Newton--Okounkov polygons) 
that many important positivity aspects can be observed infinitesimally.  The philosophy dates back at least to Demailly's work \cite{D},
where  he introduces Seshadri constants in order to capture  the local positivity of a divisor. 

These ideas were further developed in \cite{Nak1}, where Nakamaye introduced  moving Seshadri constants, 
and then generalized to a large extent in  \cite{ELMNP2}. In fact, one of the highlights of \cite{ELMNP2} is the description of the connection 
between augmented base loci and moving Seshadri constants. 
 
The goal of the current subsection is to find a similar relation between  infinitesimal Newton--Okounkov polygons and moving Seshadri constants.

Let $D$ be a big and nef $\QQ$-divisor on a smooth projective surface $X$,  $x\in X$  a closed point. The \textit{Seshadri constant} of $D$ at $x$ 
is defined to be the non-negative real number
\[
\epsilon(D;x) \ \deq \ \underset{x\in C\subseteq X}\inf \Big\{\frac{(D.C)}{\mult_x(C)}\Big\}\ ,
\]
where the infimum is taken over all reduced irreducible curves $C\subseteq X$ passing through $x$. If $\pi:X'\rightarrow X$ denotes 
the blow-up of $X$ at  $x$, then 
\[
\epsilon(D;x) \equ \max\{\epsilon>0 | \pi^*(D)-\epsilon E \textup{ is nef}\}\ . 
\]
For basic properties of Seshadri constants and further references the reader is kindly invited to consult  \cite{PAG}*{Section 5.1}.

Moving Seshadri constants were initially introduced by Nakamaye in \cite{Nak1} with the purpose of  encoding  local positivity of the $\QQ$-divisor $D$,
when it is merely big. For nef divisors moving Seshadri constants agree with   Seshadri constants as defined above.  If $x\notin\Null(D)$, 
then the \textit{moving Seshadri constant} of $D$ at $x$ is defined to be
\[
\epsilon(||D||;x) \deq \smash{\displaystyle\sup_{f^*(D)=A+E}} \epsilon(A,f^{-1}(x))\ , 
\]
where the supremum is taken over all birational morphisms $f:X''\rightarrow X$, with $X''$ smooth  
which are isomorphisms over a neighborhood of $x$, and all decompositions $f^*(D)=A+E$, with $A$ an ample $\QQ$-divisor,
$E$ effective and $x\notin\textup{Supp}(E)$. If $x\in \Null(D)$, then we put $\epsilon(||D||,x)=0$. This invariant is continuous inside the big cone, thus it is well-defined even for big $\RR$-divisors.

Suppose that $D$ is a big $\RR$-divisor on $X$ and that $x\notin\Neg(D)$. By Theorem~\ref{thm:triangle} one can introduce the following invariant for
points $y\in E$:
\[
\xi(\pi^*(D);y) \ \deq \  \sup\{\xi \geq 0 \mid \Delta_{\xi}^{-1}\subseteq \Delta_{(E,y)}(\pi^*(D))\} \ .
\]
The main goal of this subsection is to prove the following theorem connecting moving Seshadri constants to infinitesimal Newton--Okounkov polygons.

\begin{theorem}\label{thm:main2}
Let $D$ be a big $\RR$-divisor on $X$. If $x\notin\Neg(D)$, then 
\[
\epsilon(||D||;x)  \equ \xi(\pi^*(D);y)
\]
for any closed point $y\in E=\pi^{-1}(x)$.
\end{theorem}

\begin{remark}
One of the goals of \cite{ELMNP2} was to show that $\epsilon(||D||,x)\neq 0$ if and only if $x\notin\Null(D)$
(see Theorem~6.2 from the quoted paper). Based on Theorem~\ref{thm:triangle} and Theorem~\ref{thm:main2}, one can redefine moving Seshadri constant
so that it  encodes information when a point is included in the negative locus:
\begin{displaymath}
\epsilon(||D||;x) \  \deq  \  \left\{ \begin{array}{lll}
\xi(\pi^*(D), y) & \textrm{if $x\notin\Null(D)$ and any $y\in E$}\ .\\
0 & \textrm{if $x\in \Null(D)\setminus\Neg(D)$} \ .\\
-1 & \textrm{if $x\in \Neg(D)$} \ .
\end{array} \right.
\end{displaymath}
It is  important to point out  that Theorem~\ref{thm:main2} explains how  this constant can be computed directly on the blow-up of $X$ at  $x$
instead of taking into account all the blow-ups as we have seen in the original definition of moving Seshadri constant.
\end{remark}

\begin{remark}
Our proof of Theorem~\ref{thm:main2} is self-contained in the sense that it  does not make use of  non-trivial material beside the content of this article. Admittedly, it could be streamlined by relying on  the fact that  moving Seshadri constants describe the asymptotic rate of growth of jet separation at $x$, but this requires the introduction of restricted volumes, and as such goes against our intentions. 
\end{remark}

\begin{lemma}\label{lemma:function constant}
With notation as above,  the function 
\[
E \ni y \mapsto \xi (\pi^*(D);y)
\]
is constant. We denote then $\xi(\pi^*(D);y)$ by $\xi(D;x)$.
\end{lemma}
\begin{proof}
To begin with,  Remark~\ref{rem:vertical} says that  the length of the vertical segment $\Delta_{(E,y)}(\pi^*(D))_{t=\xi}$ is independent of  $y\in E$ for any $\xi\in [0,\mu']$. Furthermore, 
\[
\Delta_{(E,y)}(\pi^*(D))_{t=\xi}\dsubseteq \{\xi\}\times [0 , \xi]
\]
by Proposition~\ref{prop:propinf}. Hence whenever $\{\xi\}\times [0 , \xi]\subseteq \Delta_{(E,y)}(\pi^*(D))$ for some $y\in E$, the same holds  for all points of  $E$. Since $x\notin\Neg(D)$,  then $(0,0)\in \Delta_{(E,y)}(\pi^*(D))$ for any $y\in E$, via Theorem~\ref{thm:triangle}. Therefore $\Delta_{\xi}^{-1}\dsubseteq \Delta_{(E,y)}(\pi^*(D))$ for all $y\in E$, assuming that  the same property is known for just one point in $E$. 
Thus $\xi(\pi^*(D),y)$ does not depend on $y$. 
\end{proof}

First we move on to  give  a proof of Theorem~\ref{thm:main2} in the big and nef case. 
\begin{proposition}\label{prop:bignef}
Let $P$ be a big and nef $\RR$-divisor on $X$. Then $\epsilon(P,x) =\xi(P;x)$ for any $x\in X$. 
\end{proposition}
\begin{proof}
We verify first that $\epsilon(D;x)\leq\xi(D;x)$. By the definition of  Seshadri constants it is enough to show that if $\pi^*(P)-t E$ is nef for all $0\leq t\leq \epsilon$, then $\epsilon < \xi(\pi^*(P);y)$ for some  $y\in E$. 

Recall from Section~2.1 that 
\[ 
\Delta_{(E,y)} (\pi^*(P)) \equ \{ (t,z)\in \RR^2_+ \ | \ \nu\leq t\leq \mu, \alpha(t)\leq 
z\leq \beta(t)\}\ ,
\]
where $\alpha(t) \equ  \ord_y(N_t|_E)$, $\beta(t) \equ \alpha(t)+P_t\cdot E$, and $\pi^*(P)-tE=P_t+N_t$ is the appropriate Zariski decomposition. Note that $\pi^*(P)-tE$ is nef, and thus $N_t=0$ for all $0\leq t \leq\epsilon$. In particular, $\alpha(t)=0$ and $\beta(t)=t$. Hence $\Delta^{-1}_{\epsilon} \dsubseteq \Delta_{(E,y)}(\pi^*(P))$ and consequently, $\epsilon\leq \xi(\pi^*(P);x)$.

For the reverse inequality, we show that  if $\xi<\xi(P;x)$ then $\pi^*(P)-t E$ is nef for all 
$0\leq t\leq \xi$. By Remark~\ref{rem:shift}, then $(0,0)\in \Delta_{(E,y)}(\pi^*(D)-tE)$ for any $t\in[0,\xi]$ and all $y\in E$. Thus, Theorem~\ref{thm:main1} yields
\begin{equation}\label{eq:null}
\Neg(\pi^*(P)-t E)\ \bigcap \ E \equ \varnothing \ \ \text{for all $t\in[0,\xi]$.}
\end{equation}
We prove that this condition forces $\pi^*(P)-t E$ to be nef. Let $\pi^*(P)-t E \equ P_t+\sum a_iE^t_i$ be its  Zariski decomposition. Since $P$ is nef,  (\ref{eq:null}) implies that $((\pi^*(P)-t E)\cdot E^t_i)\geq 0$ for all $i$. On the other hand, by the construction of  Zariski decomposition we must have $(\pi^*(P)-t E)\cdot E^t_i<0$ for some $i$. Thus, each $a_i=0$ and  $\pi^*(P)-t E$ is big and nef  for each $t\in [0,\xi]$. This ends the proof.
\end{proof}

\begin{proof}[Proof of Theorem~\ref{thm:main2}]
By Proposition~\ref{prop:bignef}, it suffices  to show that
\[
 \xi(D;x) \equ \smash{\displaystyle\sup}_{f^*(D)=A+E} \{\xi(A,f^{-1}(x))\}, \textup{ whenever }x\notin\Neg(D) \ ,
\]
where the supremum is taken over all birational morphisms $f\colon X''\rightarrow X$ with $X''$ smooth
that  are isomorphisms over a neighborhood of $x$, and all decompositions $f^*(D)=A+E$, with $A$ an ample $\RR$-divisor, $E$ effective, and $x\notin\Supp(E)$. 

For a given such map $f$  it is not hard to see that $\xi(D,x) \equ \xi(f^*(D),f^{-1}(x))$ as a consequence of a stronger statement saying that $\mathcal{K}'(D,x) \equ  \mathcal{K}'(f^*(D),f^{-1}(x))$, proved in Lemma~\ref{lem:K equal} below.
 
Granting this, it only remains to show that if $D$ is a big $\RR$-divisor on $X$, then
\begin{equation}\label{eq:sup}
\xi(D,x) \equ  \sup\{\xi(A,x) \ | \ D=A+E,\ A\textup{ ample},\ E \textup{ effective and }x\notin\textup{Supp}(E)\}\ .
\end{equation}
To this end,  let $D=A+E$ be a decomposition as in $(\ref{eq:sup})$, and let $\pi\colon X'\rightarrow X$ the blow-up of $X$ at the point $x$. Since $x\notin\Supp(E)$, it follows quickly  that 
\[
\Delta_{(E,y)}(\pi^*(A)) \dsubseteq \Delta_{(E,y)}(\pi^*(D))
\]
for any point $y\in E$: namely, if $D'\equiv A$ is an effective $\RR$-divisor, then $D'+E\equiv D$ is also $\RR$-effective. Furthermore, since $x\notin\Supp(E)$, one has
\[
\nu_{(E,y)}(\pi^*(D'+E)) \equ \nu_{(E,y)}(\pi^*(D'))\ , \text{ for all } y\in E.
\] 
Using the definition of Lazarsfeld and Musta\c t\u a for Newton--Okounkov polygons of $\RR$-divisors, one obtains the inclusion of the polygons above. This proves the inequality '$\geq$' in $(\ref{eq:sup})$.

For the inequality '$\leq$' in (\ref{eq:sup}), let $D=P_D+N_D$ be the corresponding Zariski decomposition. Since $x\notin\Supp(N_D)$,  then, by Remark~\ref{rem:ZD of pullback}, we know that $\pi^*(D)=\pi^*(P_D)+\pi^*(N_D)$ is the Zariski decomposition of $\pi^*(D)$. This condition also implies that $\Supp(\pi^*(N_D))\cap E \equ \varnothing$. Thus, by Lemma~\ref{lem:positive}, we have 
\[
\Delta_{(E,y)}(\pi^*(D)) \equ \Delta_{(E,y)}(\pi^*(P_D))\ , \ \ \textup{for any } y\in E \ .
\]
This reduces our problem to case when $D=P$ is big and nef. However, by Lemma~\ref{lem:kodaira} there exists an effective divisor $E$, such that $P-\frac{1}{k}E$ is ample $\RR$-divisor for any natural $k\gg 0$. Using this and the continuity property of the Newton--Okounkov polygons inside the big cone, the direct inequality takes places when $D=P$ is big and nef, which finishes the proof of the theorem.
\end{proof}

\begin{lemma}\label{lem:K equal}
 With notation as above, 
 \[
\mathcal{K}'(D,x) \equ  \mathcal{K}'(f^*(D),f^{-1}(x))\ .
\]
\end{lemma}
\begin{proof}
Applying the ideas from the last part of the proof of Lemma~\ref{lem:positive}, i.e. Lemma~8 from \cite{AKL}, where the classes forming the limit $A_n$ are big and semi-ample, it is enough to show the statement in the case when $D$ is a big $\QQ$-divisor.

Now, by Zariski's Main theorem  pulling back sections  defines the isomorphisms 
\[
H^0(X,\sO_{X}(mD))\simeq H^0(X'',\sO_{X''}(mf^*(D)))\ 
\]
for all $m>0$. As $f$ is an isomorphism over a neighborhood of $x$, the computations of the infinitesimal Newton--Okounkov polygons on both sides of $f$ can be done on two isomorphic neighborhoods containing the corresponding exceptional divisors.
Thus, the two  sets of polygons are equal. 
\end{proof}

\subsection{Applications to questions about Seshadri constants}

In this subsection we discuss some interesting applications to questions about Seshadri constants using the material above. We start with an observation regarding valuative points on the boundary of infinitesimal Newton--Okounkov polygons.

\begin{corollary}\label{cor:ratvaluation}
Let $D$ be a big $\QQ$-divisor on $X$. Fix a point $x\notin\Null(D)$ and suppose that 
\[
\Delta^{-1}_{\xi} \dsubseteq \Delta_{(E,y_0)}(\pi^*(D))
\]
for some $\xi>0$ and $y_0\in E$. Then any rational point on the diagonal segment $\{(t,t)|0\leq t<\xi\}$ and on the horizontal segment  $[0, \xi)\times\{0\}$ is  valuative.
\end{corollary}
\begin{remark}
It is somewhat surprising that the rational points on the diagonal segment are valuative. As it was pointed out in Remark~\ref{rem:vanishing}, the reason is that the curve $E$ is rational.
\end{remark}

\begin{proof}
The statement for the horizontal line segment $(0,\xi)\times\{0\}$ can be obtained analogously as in the first part of the proof of Corollary~\ref{cor:valpoints}. Furthermore, since $x\notin\Null(D)$, then $x\notin\textup{\textbf{B}}(D)$. This implies that the origin $(0,0)$ is a valuative point.

For the points on the diagonal, let $t\in[0,\xi)$ be a rational number. Our  goal is to prove that $(t,t)$ is a valuative point. By Remark~\ref{rem:shift}, it is enough to show that $(0,t)\in\Delta_{(E,y_0)}(\pi^*(D)-tE)$
is a valuative point for the divisor $\pi^*(D)-tE$. 

By Theorem~\ref{thm:main2}, we know that $\Delta_{\xi}^{-1}\dsubseteq \Delta_{(E,y)}(\pi^*(D))$ for all $y\in E$. Thus, $\Delta_{(E,y)}(\pi^*(D)-tE)$ contains a small simplex 
for any $y\in E$, if we make use of Remark~\ref{rem:shift}. This implies via Theorem~\ref{thm:main1} that
\begin{equation}\label{eq:int}
\Null(\pi^*(D)-tE) \ \cap \ E \equ \varnothing \ .
\end{equation}
In what follows we  reduce the statement to the case of ample divisors. Let $\pi^*(D)-tE \equ P_t+N_t$ be the appropriate Zariski decomposition, and assume that all the divisors involved are integral.  By (\ref{eq:int}),  
we know that $y_0\notin\textup{Supp}(N_t)$. Thus, by Lemma~\ref{lem:positive}, 
\[
\Delta_{(E,y_0)}(\pi^*(D)-tE) \equ  \Delta_{(E,y_0)}(P_t)\ .
\]
Recall that \cite{PAG}*{Proposition 2.3.21} shows that  the inclusion map 
\[
H^0(X',sO_{X'}(mP))\rightarrow H^0(X',\sO_{X'}(m(\pi^*(D)-tE))\ ,
\]
defined by the multiplication by the divisor $mN_t$, is an isomorphism. Hence, we reduced the problem to the case when $\pi^*(D)-tE=P_t$ is big and nef and the point of interest is $(0,t)=(0,(P_t.E))$. 

By Remark~\ref{rem:ZD of pullback}, there exist irreducible curves $C_i\subseteq \Null(P_t)$ and rational numbers $\epsilon_i>0$ for $i=1,\ldots ,k$, such that $A_t\deq P_t-\sum_{i=1}^{i=k}\epsilon_iC_i$ is an ample $\QQ$-divisor. By (\ref{eq:int}), we have that $C_i\cap E=\varnothing$ for all $i=1,\ldots, k$. Thus the point $(0,t)=(0,(A_t.E))$ belongs to $\Delta_{(E,y_0)}(A_t)$, and it suffices to treat the case of ample 
$\QQ$-divisors. However,  this situation has  already been discussed in Remark~\ref{rem:vanishing}, which  finishes the proof.
\end{proof}

As a consequence, we obtain  criteria for finding  lower bounds for Seshadri constants. 
 
\begin{corollary}\label{cor:lowerbound}
Let $X$ be a smooth surface, $x\in X$ a point,  $A$  an ample $\QQ$-divisor on $X$ and $q>0$ a rational number. 
Then the following conditions are equivalent:
\begin{enumerate}
\item The Seshadri constant $\epsilon(A,x)\geq q$.
\item There exists a point $y\in E$ such that $\Delta_{(E,y)}(\pi^*(A))$ contains both the point $(q,0)$ and $(q,q)$.
\item For all $y\in E$, $(q,0)\in \Delta_{(E,y)}(\pi^*(A))$.
\item For all $y\in E$, $(q,q)\in \Delta_{(E,y)}(\pi^*(A))$.
\item There exists $y_1\neq y_2\in E$ so that the point $(q,q)\in \Delta_{(E,y_i)}(\pi^*(A))$ for any $i=1,2$.
\end{enumerate}
\end{corollary}

\begin{remark}\label{rem:tangency}
As  $q\in\QQ$,  Corollary~\ref{cor:ratvaluation} implies that  the conditions in the statement can be translated to ones about linear series on $X$ itself. For example, the point $(q,0)$ lies in $\Delta_{(E,y)}(\pi^*(A))$, provided there exists an effective $\QQ$-divisor $D'\sim_{\QQ}A$ for which  
$\mult_x(D')\geq q$ and the tangent direction of each branch of $D'$ is distinct from the one defined by  $y\in E$. 

Similarly, the point $(q,q)\in \Delta_{(E,y)}(\pi^*(A))$ whenever there exists an effective $\QQ$-divisor $D''\sim_{\QQ}A$ such that $\mult_x(D'')\geq q$ and the tangent direction of each branch of $D''$ is the same as the one defined by the point $y$.
\end{remark}

\begin{remark}
Suppose that $D', D''\sim_{\QQ}A$ are effective divisors having the property that  
\[
\min\{\ord_x(D'),\ord_x(D'')\}\dgeq q
\]
and any tangent direction of any branch of $D'$ is distinct from those  of $D''$. Then $\epsilon(A,x)\geq q$. 

We can see this as follows. By Corollary~\ref{cor:lowerbound}, it suffices to show $(3)$. However, note that 
$\nu_{(E,y)}(\pi^*(D')) \equ (q,0)$ for any $y\in E$ that is different from the tangent direction of any branch of $D'$. As the same can be said about $D''$, the condition we asked above indeed implies $(3)$.
\end{remark}
\begin{proof}
For starters,  we observe that $(1)$ implies conditions $(2)-(5)$ via Theorem~\ref{thm:main2}. Therefore we are left with proving the reverse implications. 
 
First, $(2)$ implies $(1)$  follows again from Theorem~\ref{thm:main2} and the fact that $(0,0)\in\Delta_{(E,y)}(\pi^*(A))$, as $A$ is ample. Notice that $(4)$  implies $(5)$ is immediate, and $(3)$ implies $(1)$ follows word by word from the second part of the proof of 
Proposition~\ref{prop:bignef}.

We are left to prove $(5)\Longrightarrow (1)$. Fix $t\in (0,q)$ and the goal is to show that $\pi^*(A)-tE$ is nef. Let $\pi^*(A)-tE=P+N$ be the corresponding Zariski decomposition. Since both points $(0,0)$ and $(q,q)$ are contained in $\Delta_{(E,y_i)}(\pi^*(A))$ for any $i=1,2$, by convexity the point $(t,t)$ is also contained in these polygons. By the formula for  Newton--Okounkov polygons from Section~2.1, this implies that 
\[
\ord_{y_i}(N|_E)+(P.E) \equ t \ \ \text{ for any $i=1,2$.}
\]
On the other hand, $(\pi^*(A)-tE). E = (P+N).E = t$. In particular, 
\[
(N\cdot E) \equ \ord_{y_1}(N|_E) \equ \ord_{y_2}(N|_E)\ .
\]
By the same token as in Remark~\ref{rem:definition},  the first equality implies that the effective divisors $N$ and $E$ intersect only at $y_1$, 
while the second one implies that they intersect only at $y_2$. Since $y_1\neq y_2$, necessarily $N=0$, i.e. $\pi^*(A)-tE$ is nef. 
This finishes the proof.
\end{proof}

\section{Applications}

We present some  applications to questions regarding Seshadri constants seen through the lenses of the theory of Newton--Okounkov polygons developed in the previous sections. 

First, we give a new proof of a lower bound for very generic points by Ein and Lazarsfeld that relied originally on deformation theory; our argumentation is based on earlier work of Nakamaye. Second, based on Theorem~\ref{thm:main1}, we introduce a new invariant  that encodes the size of the largest simplex that can be included in some Newton--Okounkov polygon of a given  divisor by varying the curve flag. We connect this invariant to the Seshadri constant. Lastly, using Diophantine approximation, we show that whenever the surface has a rational polyhedral nef  cone, the global Seshadri constant at any point is strictly positive. 

\subsection{Generic infinitesimal Newton--Okounkov polygon.}
Let $A$ be an ample Cartier divisor on a smooth projective surface $X$.  Ein and Lazarsfeld proved in \cite{EL} that $\epsilon(A,x)\geq 1$ 
for very general point $x\in X$. Later, Cascini and Nakamaye in \cite{CN}, gave a different proof avoiding deformation theory 
based on  ideas  developed previously by the second author in \cite{Nak2}. Here we translate the line of thought of Cascini--Nakamaye 
to the language of infinitesimal Newton--Okounkov polygons.

The main extra ingredient is the following observation of  Nakamaye (see  \cite{Nak2}*{Lemma 1.3}). As Nakamaye points out, the result  is an easy consequence of a statement about the smoothing divisors in families as seen in \cite{PAG}*{Proposition 5.2.13}. This claim initially  appears  in \cite{Nak2}, and it is used both in \cite{Nak2} and \cite{CN} to establish lower bounds on Seshadri constants in higher dimensions.

\begin{lemma}\label{lem:nakamaye}
Let $x\in X$ be a very general point and $D$ be an effective integral divisor on $X$. Suppose $W\subseteq X$ be an irreducible curve passing through $x$. Let $\overline{W}$ be the proper transform of $W$ through the blow-up $\pi :X'\rightarrow X$ of the point $x$. Also, define 
\[
\alpha(W) \ = \ \textup{inf}_{\beta\in \QQ}\{\overline{W}\subseteq\Null(\pi^*(D)-\beta E)\}\ .
\]
Then $\textup{mult}_{\overline{W}}(||\pi^*(D)-\beta E||)\geq \beta -\alpha(W)$ for all $\beta\geq \alpha(W)$.
\end{lemma} 

Lemma~\ref{lem:nakamaye} forces the generic infinitesimal Newton--Okounkov polygon of very generic points to land in certain area of the plane, depending on the Seshadri constant. 

\begin{proposition}\label{prop:genericinf}
Let $A$ be an ample Cartier divisor on $X$ and let $x\in X$ be a very general point. Then the following mutually exclusive cases can occur. 
\begin{enumerate}
\item $\mu'(A,x)=\epsilon (A,x)$, then $\Delta(A,x)=\Delta^{-1}_{\epsilon(A,x)}$.
\item $\mu'(A,x)>\epsilon (A,x)$, then  there exists an irreducible curve $C\subseteq X$ with $(A\cdot C)=p$ and $\mult_x(C)=q$ such that $\epsilon(A,x)=p/q$. 
Under these circumstances, 
\begin{enumerate}
\item Whenever $q\geq 2$, $\Delta(A,x)\subseteq \triangle_{ODB}$, where $O=(0,0), D=(p/q,p/q)$ and $B=(p/(q-1),0)$.
\item Whenever $q=1$, the polygon $\Delta(A,x)$ is contained in the area  below the line $y=t$ and between the horizontal lines 
$y=0$ and $y=\epsilon(A,x)$.
\end{enumerate}
\end{enumerate}
\end{proposition}

\begin{corollary}\label{cor:genericseshadri}
 Let $A$ be an ample line bundle on a smooth projective surface. Then $\epsilon(A,x)\geq 1$ for very generic points $x\in X$.
\end{corollary}
\begin{proof}
 By the definition of Seshadri constants and Proposition~\ref{prop:genericinf}, it suffices  to consider the case $(2a)$. 
 Thus,  we know that $\Delta(A;x)\subseteq \triangle_{ODB}$, and as a consequence 
\[
\textup{Area}(\Delta(A;x)) \equ  \frac{A^2}{2} \ \leq \ \textup{Area}(ODB) \equ \frac{p^2}{2q(q-1)} \ . 
\]
In particular,  $\epsilon(A,x)\geq \sqrt{(A^2)(1-\frac{1}{q})}$. Hence,  if we assume  $\epsilon(A,x)<1$, then  by  the rationality of $\epsilon(A;x)$, we also have  $\epsilon(A,x)\leq \frac{q-1}{q}$. Using the inequality between the areas, we arrive at  $(A^2)<1$, which stands in contradiction with the assumption
that $A$ is an ample Cartier divisor. 
\end{proof}

\begin{proof}[Proof of Proposition~\ref{prop:genericinf}]
If  $\epsilon(A,x)=\mu'(A,x)$, then one has automatically $\Delta=\Delta_{\mu'(A,x)}^{-1}$. Therefore we can assume without loss of generality that $\mu'(A,x)>\epsilon(A,x)$. In particular, there exists a curve $C\subseteq X$ with $(A\cdot C)=p$ and $\mult_x(C)=q$ such that $\epsilon(A,x)=p/q$. 

Let $\overline{C}$ be the proper transform of  $C$ on $X'$. The idea of the proof is to calculate  the length of the  vertical segment in the polygon $\Delta(A,x)$ at $t=t_0$ for any $t_0\geq \epsilon(A,x)$:
\[
\textup{length}\big(\Delta(A)_{t=t_0}\big) \equ  (P_{t_0}\cdot E)\ = \ t_0-(N_{t_0}\cdot E)\ ,
\]
where $\pi^*(A)-t_0E=P_{t_0}+N_{t_0}$ is the corresponding Zariski decomposition. By Lemma~\ref{lem:nakamaye}, one can write $N_{t_0} \equ (t_0-\epsilon(A,x))\overline{C}+N_{t_0}'$, where $N_{t_0}'$ remains  effective. This implies the following inequality
\begin{equation}\label{eq:length}
\textup{length}\big(\Delta(A)_{t=t_0}\big) \equ t_0-\big((t_0-\epsilon(A,x))\overline{C}+N_{t_0}' \cdot E\big) \ \leq \ 
t_0-(t_0-\epsilon(A,x))q \ .
\end{equation}
The vertical line segment $\Delta(A;x)_{t=t_0}$ starts on the $t$-axis at the point $(t_0,0)$ by Proposition~\ref{prop:propinf} for any $t_0\geq 0$. Therefore,  by (\ref{eq:length}), the polygon sits below the line 
\[
y \equ t_0-(t_0-\epsilon(A,x))q \equ  (1-q)t_0+\epsilon(A,x)q \ .
\]
When $q=1$, then this line is the horizontal line $y=\epsilon(A,x)$ and when $q\geq 2$ then it is the line passing through the points $D=(p/q,p/q)$ and $B=(p/(q-1),0)$. This finishes the proof.
\end{proof}

We conclude this subsection with a lower bound on Seshadri constants of quintic surfaces. 
While the bound might be known to experts,  we include it here  here since it is an illustration of the use of infinitesimal Newton--Okounkov bodies. 

\begin{example}\label{ref:quintic}
Inspired by the work of Nakamaye,  we  show that for any smooth quintic surface $X\subseteq\PP^3$, 
if  $A$ is the line bundle defining the embedding, then we have  $\epsilon(A;x)\geq 2$ for a very generic point $x\in X$.

The main ingredient  is Proposition~\ref{prop:genericinf}; suppose that $\epsilon(A;x)<2$. Then there exists an irreducible curve $C\subseteq X$
containing the point $x$  such that $\mult_x(C)=q$, $(A\cdot C)=p$ and $\epsilon(A;x)=\frac{p}{q}$. 

If $q=1$, then $p=1$ as well, and this implies that through a very general point of $X$ there passes a line. This forces $X$ to be uni-ruled, 
which is  not the case for  quintic surfaces.

Thus we can assume $q\geq 2$. Then the generic infinitesimal Newton--Okounkov polygon $\Delta(A,x)$ is contained in the triangle $\triangle_{ODB}$, 
where $O=(0,0)$, $D=(p/q,p/q)$, and $B=(p/(q-1),0)$  by Proposition~\ref{prop:genericinf}. As a consequence we have the following inequality
\[
\textup{Area}(\triangle_{ODB}) \equ  \frac{p}{q}\cdot \frac{p}{q-1} \ \dgeq \ \textup{Area}(\Delta(A,x)) \equ 5 \ .
\]
If $q\geq 5$, then this yields $p\geq 2q$, which contradicts our initial assumption that $\epsilon(A;x)<2$. 
On the other hand, if $2\leq q\leq 4$, then we obtain that the list of  remaining choices to tackle is  
\[
\frac{p}{q} \equ \frac{2}{2},\frac{3}{3}, \frac{4}{4}, \frac{3}{2}, \frac{4}{3}, \frac{5}{3},\frac{5}{4},\frac{7}{4}\ ,
\]
since $p<2q$. However, none of these pairs  satisfy the area inequality above, hence  we are done.

The same line of thought  implies  that whenever $\textup{Pic}(X)=\ZZ A$, then $\epsilon(A;x)=2$ for a very generic point $x\in X$ 
if and only if there is a curve $C\in |2A|$ with the property that $\mult_x(C)=5$.
\end{example}

\subsection{The largest simplex constant}

It was established in  Section~3 that all Newton--Okounkov polygons of  ample line bundles contain a standard simplex of some  size 
that depends on the choice of the flag. If the curve in the flag is chosen to be very positive, the sizes of these standard simplices can become
arbitrarily small. Thus, the exciting question to ask is how large they can become. 

\begin{definition}(Largest simplex constant)
Let $A$ be an ample $\QQ$-divisor on $X$ and let $(C,x)$ be an admissible flag. We define  
\[
\lambda(A;C,x) \deq  \sup\{\lambda\geq 0 \mid \Delta_{\lambda}\subseteq \Delta_{(C,x)}(A)\}\ .
\]
The \emph{largest simplex constant} of $A$ at the point $x$ is defined to be
\[
\lambda(A;x) \deq \sup \{\lambda(A;C,x) \mid C\subseteq X\textup{ is an irreducible curve that is smooth at } x \}\ .
\]
\end{definition}

\begin{remark}
Not unexpectedly, one can define the largest simplex constant for big divisors in general, assuming that the point $x$ is not contained in the null locus of the divisor. All formal properties of $\lambda(A;x)$ go through almost verbatim, hence the details are left to the (interested) reader. 
\end{remark}

The goal of this subsection is to relate the largest simplex constant to  Seshadri constants. 

\begin{proposition}\label{prop:simplexseshadri}
With the notation as above, $\epsilon(A;x)\geq \lambda(A;x)$.
\end{proposition}

\begin{remark}\label{rem:miranda}
Proposition~\ref{prop:simplexseshadri} implies that there is no uniform lower bound on the largest simplex constant holding at every point of every surface. This follows from the non-existence of the analogous bound for Seshadri constants
 as seen in Miranda's example in \cite{PAG}*{Example 5.2.1}. 
\end{remark}

\begin{remark}\label{rmk:fake}
It is a  natural question after Proposition~\ref{prop:simplexseshadri}  whether there are examples with $\lambda(A;x)\neq \epsilon(A;x)$. One such example, is Mumford's  fake projective plane (for the actual construction see \cite{M}).

The surface $X$ is of general type with ample canonical class $K_X$,  $(K_X^2)=9$, and geometric genus $p_g=H^0(X,\sO_X(K_X))=0$. Since $\textup{Pic}(X)=\ZZ H$, these conditions imply that $H^0(X,\sO_X(H))=0$. This means that whenever $(C,x)$ is an admissible flag, we have $C\in |dH|$ with $d\geq 2$, hence  clearly  $\lambda(H,x)\leq 1/2$ for any $x\in X$. On the other hand, we know by Corollary~\ref{cor:genericseshadri} that $\epsilon(H,x)\geq 1$ when $x\in X$ is a 
very general point.
\end{remark}
\begin{proof}[Proof of Proposition~\ref{prop:simplexseshadri}]
Theorem~\ref{thm:main1} yields that $\lambda \deq  \lambda (A;C;x)>0$ for any admissible flag $(C,x)$. By fixing the flag $(C,x)$, 
it is enough to show that $\epsilon(A,x)\geq \lambda$.

By Corollary~\ref{cor:ratvaluation} there exist  sequences of real numbers $\epsilon^v_n$ and $\epsilon^h_n$ with both $\lambda-\epsilon^v_n$ and $\lambda-\epsilon^h_n$ rational, and  sequences of effective $\QQ$-divisors $(D^v_n)$ and $(D^h_n)$ with $D^v_n,D^h_n=_{\textup{num}}D$ for any $n\in\NN$ such that 
\[
\nu_{(C,x)}(D^v_n) \equ (0,\lambda-\epsilon^v_n) \textup{ and } \nu_{(C,x)}(D^h_n) \equ (\lambda-\epsilon^h_n, 0) \ .
\]
This yields  $C\nsubseteq \Supp (D^v_n)$ for any $n\in\NN$, and by Remark~\ref{rem:definition}, we obtain  that 
\begin{equation}\label{eq:in1}
(D\cdot C) \ = \ \frac{(D^v_n\cdot C)}{\mult_x(C)}\geq \ \lambda -\epsilon^v_n\ ,
\end{equation}
where we took into account that $C$ is smooth at $x$. By looking at  the valuation vector of $D^h_n$,
we can write $D^h_n=(\lambda-\epsilon^h_n)C+N_n$, where $N_n$ is effective and $\mult_x(N_n)=0$. Thus, for any irreducible curve $F\neq C$ passing through $x$, we have that $F\nsubseteq\Supp (D^h_n)$. As a consequence we get the following string of inequalities
\begin{equation}\label{eq:in2}
\frac{(D\cdot F)}{\mult_x(F)} \equ \frac{(D^h_n\cdot F)}{\mult_x(F)}\geq \frac{(\lambda-\epsilon^h_n)(C\cdot F)}{\mult_x(F)} \geq 
\lambda-\epsilon^h_n\ ,
\end{equation}
where the last  inequality follows from the fact that $(C\cdot F)\geq \mult_x(F)\cdot \mult_x(C)$ whenever $F\neq C$. 

Observing the definition of Seshadri constants, and taking  the limit in both equations~(\ref{eq:in1}) and~(\ref{eq:in2}), we arrive at 
$\epsilon(D;x)\geq \lambda$, as required. 
\end{proof}

\subsection{Diophantine approximation.}
Here we show via Diophantine approximation that the largest simplex constant of a surface is strictly positive whenever it has a rational polyhedral nef cone. It is important to note that the semigroup of ample line bundles of $X$ is not necessarily finitely generated even if the nef cone is rational polyhedral: the  lattice semigroup   $\NN^2\cap \RR^2_{>0}$ is one such example. Furthermore, the line bundles sitting on the boundary of the nef cone might not even have sections asymptotically as seen in examples provided 
in \cite{Ott}.

It was Nadel who first stressed the relevance of Diophantine approximation to local positivity issues (see \cite{EKL}). 
This train of thought was further explored by Nakamaye. Very  recently a  deep connection between Diophantine approximation and Seshadri constants 
was established by McKinnon and Roth in \cite{MR}.

\begin{theorem}\label{thm:polyhedral}
Let $X$ be an irreducible projective variety with  a rational polyhedral nef cone. Then there exists a natural number $m>0$ such that the linear series $|mA|$ is base-point free for any ample Cartier divisor $A$ on $X$.
\end{theorem}

\begin{remark}
Theorem~B follows easily as a consequence of Theorem~\ref{thm:polyhedral} and Proposition~\ref{prop:simplexseshadri}. Furthermore, the above theorem implies that whenever $X$ is a smooth projective variety with a rational polyhedral nef cone, then there exists a strictly positive constant $\epsilon (X)>0$ such that $\epsilon(A;x)\geq \epsilon(X)$ for any $x\in X$ and any ample Cartier divisor $A$ on $X$. This is due to the fact that whenever $B$ is an ample and base-point free divisor, then $\epsilon(B;x)\geq 1$ for any $x\in X$.
\end{remark}

We will need the following statement during the proof. 

\begin{lemma}\label{lem:wilson}
Let $X$ be an irreducible projective variety. Then there exists a Cartier divisor $B$ such that the divisor $B+P$ is base-point free 
for any nef Cartier divisor $P$ on $X$.
\end{lemma}

\begin{remark}
When $X$ is a smooth projective variety, one can be more specific about the divisor $B$. By the Anghern--Siu's theorem, the divisor $K_X+n(n+1)/2A+A+P$ defines a base-point free linear series for any ample $A$ and nef $P$. Thus, $B$ can be taken to be $K_X+n(n+1)/2A+A$ for instance.

In the general case, assuming that one does not  need a specific $B$, then one obtains Lemma~\ref{lem:wilson} by making use of 
Fujita's  vanishing theorem and Castelnuovo-Mumford regularity as in \cite{PAG}*{Theorem 2.3.9}.
\end{remark}

\begin{proof}
Suppose that $H$ is a very ample Cartier divisor on $X$. We know by Fujita's vanishing theorem, see Theorem~1.4.35 from \cite{PAG}
that there exists $m_0>0$ such that 
\[
H^0(X,\sO_X(mH+P)) \equ 0, \text{ for all } i>0, m\geq m_0 \ ,
\]
and any nef divisor $P$ on $X$. Set $B\deq (\dim (X)+1+ m_0)H$, then $H^i(\sO_X(B-iH+P))=0$ for any $1\leq i\leq \dim (X)$ and any nef $P$. 
This implies that the line bundle $\sO_X(B+P)$ is $0$-regular with respect to $H$ by \cite{PAG}*{Definition 1.8.4}. Applying 
\cite{PAG}*{Theorem 1.8.5} we obtain  that the line bundle  $\sO_X(B+P)$ is globally generated for any nef Cartier divisor $P$ on $X$.
\end{proof}

\begin{proof}[Proof of Theorem~\ref{thm:polyhedral}]
Note first that in the language of cones, Lemma~\ref{lem:wilson} says that there exists a nef divisor $B$ so that any Cartier divisor 
whose class lands in the pointed cone $B+\textup{Nef}(X)_{\RR}$, defines a base-point free linear series. In particular, the statement 
follows provided we can prove  that there exists a constant $m>0$ such that $mA\in B+\textup{Nef}(X)_{\RR}$ for any ample Cartier divisor $A$ on $X$.

This reduces the problem to a question about convex cones. So, let $f:\textup{N}^1(X)_{\RR}\rightarrow \RR^{\rho}$ be a  bijective linear map 
whose matrix has integral entries, i.e. $f(L)\in\ZZ^{\rho}\subseteq\RR^{\rho}$ for any class $L$ given by some Cartier divisor on $X$. 
Denote by $\mathcal{C}=f(\textup{Nef}(X)_{\RR})$ and $b=f(B)$. Then it suffices to check  that there exists a natural number $m>0$ having the property
that $m\xi\in b+\mathcal{C}$ for any $\xi\in\textup{int}(\mathcal{C})\cap \ZZ^{\rho}$.

Let $H\subseteq \RR^{\rho}$ be a hyperplane given by an equation with  integral coefficients, that is, we assume that there exists a vector 
$u\in\ZZ^{\rho}$ such that $H=\{x\in\RR^{\rho}| <x,u>=0\}$. If $P\notin H$, then the distance from $P$ to the hyperplane $H$ is given by the formula
\[
\textup{distance}(P,H) \ = \ \frac{|\langle P,u\rangle|}{||u||}\ .
\]
If we ask for  $P\in\ZZ^{\rho}$, then $|\langle P,u\rangle |\geq 1$, and in particular $\textup{distance}(P,H) \geq  1/||u||$. 
This yields that there exists a constant $c>0$ such that $\textup{distance}(P,H)\geq c$ for any integral point $P\notin H$. 

Going back to our setup, the conditions in the statement imply that the cone $\mathcal{C}\subseteq\RR^{\rho}$ is rational polyhedral, i.e. 
the support hyperplanes for each face are given by an equation with integral coefficients. Thus there exists a constant $c>0$ such that 
\[
\textup{distance}(P,\partial \mathcal{C}) \ \geq \ c, \textup{ for any point } P\in\textup{int}(\mathcal{C})\cap\ZZ^{\rho}\ ,
\]
where $\partial\mathcal{C}$ denotes  the boundary in $\RR^{\rho}$ of the cone $\mathcal{C}$.

Pick  $P\in\textup{int}(\mathcal{C})$, and let $\Lambda$ be the plane determined  by the origin $\mathbf{0}=(0,\ldots 0)\in \RR^{\rho}$, $b$ and $P$. 
Let $\mathcal{C}_{\Lambda}=\mathcal{C}\cap \Lambda$. This is a cone in $\RR^2$, thus it is generated by two rays $\RR_+ l_1$ and $\RR_+ l_2$, 
where both $l_1$ and $l_2$ can be taken to be rational vectors since the boundary $\partial\mathcal{C}$ is supported by rational hyperplanes,
and the plane $\Lambda$ is also defined by an equation with rational coefficients. 

Furthermore, the set $(b+\mathcal{C})\cap \Lambda$ is the cone $\RR_+l_1+\RR_+l_2$ shifted by $b$. Without loss of generality, suppose that 
the ray $\RR_+OP$ intersects first the half line $b+\RR_+l_1$ at point $D$. Then, taking into account what was said above, it is enough to find $C>0$,
that does not depend on the choice of the point $P$, so that the quotient $||OD||/||OP||<C$. Using similar triangles arguments, one has
\[
\frac{||OD||}{||OP||} \equ \frac{\textup{distance}(D,l_1)}{\textup{distance}(P,l_1)} \ \leq \ \frac{||OB||}{c} \ ,
\]
where the latter inequality follows from the Diophantine approximation statement we proved above. This finishes the proof.
\end{proof}


\raggedright

\end{document}